\documentclass[12pt, reqno]{amsart}
\usepackage{amsmath}
\usepackage{amssymb}
\usepackage{amsthm}
\usepackage{xcolor}
\usepackage{mathrsfs}
\usepackage[abbrev]{amsrefs}
\usepackage[utf8]{inputenc}
\usepackage{enumitem}
\usepackage{bm}
\usepackage{bbm}
\usepackage{graphicx}

\newtheorem{thm}{}[section]
\newtheorem{theorem}[thm]{Theorem}

\newtheorem{lemma}[thm]{Lemma}
\newtheorem{proposition}[thm]{Proposition}

\theoremstyle{definition}
\newtheorem{definition}[thm]{Definition}
\theoremstyle{remark}
\newtheorem{remark}[thm]{Remark}

\newtheorem{example}[thm]{Example}

\numberwithin{equation}{section}
\allowdisplaybreaks

\newcommand{\LPZC}{\ensuremath{\mathscr{S}}}
\newcommand{\BCLTC}{\ensuremath{\mathscr{Z}}}
\newcommand{\UCB}{\ensuremath{\mathscr{U}_S}}
\newcommand{\UTAP}{\ensuremath{\mathscr{U}}}
\newcommand{\nn}{\ensuremath{{\bm{n}}}}
\newcommand{\ls}{\ensuremath{d^{\,0}}}
\newcommand{\VV}{\ensuremath{\mathbb{V}}}

\newcommand{\Jt}{\ensuremath{\mathcal{J}}}
\newcommand{\Nt}{\ensuremath{\mathcal{N}}}
\newcommand{\Mt}{\ensuremath{\mathcal{M}}}
\newcommand{\Ts}{\ensuremath{\mathcal{T}}}
\newcommand{\TT}{\ensuremath{\mathbb{T}}}

\newcommand{\UB}{\ensuremath{\mathcal{U}}}
\newcommand{\XB}{\ensuremath{\mathcal{X}}}
\newcommand{\YB}{\ensuremath{\mathcal{Y}}}
\newcommand{\XX}{\ensuremath{X}}
\newcommand{\YY}{\ensuremath{Y}}
\newcommand{\xx}{\ensuremath{\bm{x}}}
\newcommand{\xs}{\ensuremath{\overline{\bm{x}}}}
\newcommand{\HB}{\ensuremath{\mathcal{H}}}

\newcommand{\Id}{\ensuremath{\mathrm{Id}}}
\newcommand{\Ind}{\ensuremath{\mathbbm{1}}}
\newcommand{\NN}{\ensuremath{\mathbb{N}}}

\newcommand{\FF}{\ensuremath{\mathbb{F}}}
\newcommand{\Fou}{\ensuremath{\mathcal{F}}}

\newcommand{\yy}{\ensuremath{\bm{y}}}
\newcommand{\ee}{\ensuremath{\bm{e}}}
\newcommand{\EE}{\ensuremath{\mathcal{E}}}
\newcommand{\uu}{\ensuremath{\bm{u}}}
\newcommand{\vv}{\ensuremath{\bm{v}}}
\newcommand{\zz}{\ensuremath{\bm{z}}}
\newcommand{\supp}{\operatorname{supp}}
\newcommand{\co}{\operatorname{co}}
\newcommand{\LL}{{\ensuremath{\mathcal{L}}}}

\newcommand{\ww}{\ensuremath{\bm{w}}}

\newcommand\subsetsim{\mathrel{%
\ooalign{\raise0.2ex\hbox{$\subset$}\cr\hidewidth\raise-0.8ex\hbox{\scalebox{0.9}{$\sim$}}\hidewidth\cr}}}

\AtBeginDocument{
\def\MR#1{}
}

\begin{document}

\title[]{Uniqueness of unconditional basis of infinite direct sums of quasi-Banach spaces}
\author[F. Albiac]{F. Albiac}
\address{Department of Mathematics, Statistics and Computer Sciences, and InaMat$^{2}$\\ Universidad P\'ublica de Navarra\\
Pamplona 31006\\ Spain}
\email{fernando.albiac@unavarra.es}

\author[J. L. Ansorena]{J. L. Ansorena}
\address{Department of Mathematics and Computer Sciences\\
Universidad de La Rioja\\
Logro\~no 26004\\ Spain}
\email{joseluis.ansorena@unirioja.es}

\subjclass[2010]{46B15, 46B20, 46B42, 46B45, 46A16, 46A35, 46A40, 46A45}

\keywords{uniqueness, unconditional basis, equivalence of bases, quasi-Banach space, Banach lattice}

\begin{abstract}
This paper is devoted to providing a unifying approach to the study of the uniqueness of unconditional bases, up to equivalence and permutation, of infinite direct sums of quasi-Banach spaces. Our new approach to this type of problem permits to show that a wide class of vector-valued sequence spaces have a unique unconditional basis up to a permutation. In particular, solving a problem from \cite{AlbiacLeranoz2011} we show that if $\XX$ is quasi-Banach space with a strongly absolute unconditional basis then the infinite direct sum $\ell_{1}(\XX)$ has a unique unconditional basis up to a permutation, even without knowing whether $\XX$ has a unique unconditional basis or not. Applications to the uniqueness of unconditional structure of infinite direct sums of non-locally convex Orlicz and Lorentz sequence spaces, among other classical spaces, are also obtained as a by-product of our work.
\end{abstract}

\thanks{F. Albiac acknowledges the support of the Spanish Ministry for Science and Innovation under Grant PID2019-107701GB-I00 for \emph{Operators, lattices, and structure of Banach spaces}. F. Albiac and J.~L. Ansorena acknowledge the support of the Spanish Ministry for Science, Innovation, and Universities under Grant PGC2018-095366-B-I00 for \emph{An\'alisis Vectorial, Multilineal y Aproximaci\'on}.}

\maketitle

\section{Introduction and background}
\noindent

Given a Banach space (or, more generally, a quasi-Banach space) $\XX$ with a normalized unconditional basis $(\xx_{n})_{n=1}^{\infty}$, let us write $\XX\in\UTAP$ if every normalized unconditional basis of $\XX$ is equivalent to a permutation $(\xx_{\pi(n)})_{n=1}^{\infty}$ of the basis $(\xx_{n})_{n=1}^{\infty}$. If we impose a stronger uniqueness property, where it is required that $\pi$ be the identity, we write $\XX\in\UCB$. Notice that $X\in\UCB$ if and only if $X$ has a symmetric basis and $X$ belongs to $\UTAP$. In the context of Banach spaces it is well known that $\XX\in\UCB$ if and only if $\XX$ is isomorphic to one of the spaces from the set $\LPZC=\{c_0,\ell_1,\ell_2\}$ (\cites{LinPel1968,LinZip1969}). However, for quasi-Banach spaces which are not Banach spaces the situation is quite different since there is a wide class of non-locally convex Orlicz sequence spaces, including the spaces $\ell_p$ for $0<p<1$, which belong to $\UCB$
(\cite{Kalton1977}).

Bourgain et al.\ studied in \cite{BCLT1985} the class $\BCLTC$ of those Banach spaces which can be obtained by taking the infinite direct sum of a space from $\LPZC$ in the sense of a space also in $\LPZC$, and gave a complete description of the class $\BCLTC\cap \UTAP$ by proving that the spaces $c_{0}(\ell_{1})$, $\ell_{1}(c_{0})$, $c_{0}(\ell_{2})$ and $\ell_{1}(\ell_{2})$ belong to $\UTAP$, while $\ell_{2}(\ell_{1})$ and $\ell_{2}(c_{0})$ do not. Many of the questions the authors formulated in their 1985 \emph{Memoir} remain open as of today. They conjectured that if a Banach space $\XX$ belongs to $\UTAP$ then so does the iterated copy of $\XX$ in the sense of one of the spaces from $\LPZC$. This conjecture was disproved in the general case in 1999 by Casazza and Kalton, who showed that Tsirelson's space $\Ts\in \UTAP$ whereas $c_{0}(\Ts)\notin\UTAP$ (\cite{CasKal1999}). Casazza and Kalton's work gave thus continuity to a research topic that was central in Banach space theory in the 1960's and 1970's, but that was interrupted after the \emph{Memoir}. Perhaps the researchers felt discouraged to put effort into a subject that required the discovery of novel tools in order to make headway, with little hope for attaining a satisfactory classification of the Banach spaces belonging to $\UTAP$.

At the same time, the positive results on uniqueness of unconditional basis obtained in the context of non-locally convex quasi-Banach spaces motivated further study with a number of authors contributing to the development of a coherent theory. An important advance was the paper \cite{KLW1990} by Kalton et al.\, followed by the work of Ler\'anoz \cite{Leranoz1992}, who proved that $c_{0}(\ell_{p})\in \UTAP$ for all $0<p<1$, and Wojtaszczyk \cite{Woj1997}, who proved that the Hardy space $H_{p}(\TT)$ also belongs to the class $\UTAP$ for $0<p<1$. Subsequently, it was proved that $\ell_{p}(\ell_{2})$, $\ell_{p}(\ell_{1})$, and $\ell_{1}(\ell_{p})$ also belong to $\UTAP$ for all $0<p<1$ (\cites{AlbiacLeranoz2002, AKL2004}), and the question arose of what can be said about infinite direct sums of other quasi-Banach spaces. Our aim in this paper is to fill this gap in the literature. To that end, we develop a new set of techniques which combined with reinterpretations of the already existing methods permit to obtain a myriad of new additions to the list of spaces with a unique unconditional basis.

The article is structured in five more sections. Section~\ref{sect:term} gathers the terminology and the notation that are more heavily used.
 Section~\ref{sect:preliminary} is preparatory but becomes instrumental in what follows. We survey the techniques developed by the specialists in their study of the uniqueness of unconditional structure which will be of interest for us, and give them a quantitative twist. In particular we further the study of strongly absolute bases. Section~\ref{sect:SASense} addresses the uniqueness of uconditional basis of infinite direct sums of quasi-Banach spaces in the sense of an atomic quasi-Banach lattice whose unit vector system is strongly absolute, while in Section~\ref{sect:l1Sense} we concentrate in $\ell_1$-sums of quasi-Banach spaces with strongly absolute bases. A brief digression could help the reader to understand better our approach in these theoretical sections. An infinite direct sum $\XX=(\bigoplus_{j=1}^{\infty} \XX_{j})_{\LL}$ of quasi-Banach spaces $(\XX_{j})_{j=1}^{\infty}$ in the sense of some quasi-Banach lattice $\LL$ may be regarded as an infinite matrix whose $j$th row is occupied by the vectors in $X_{j}$. Since the spaces $\XX_{j}$ come with a basis $\XB_{j}$, the vectors in $\XX_{j}$ are sequences of scalars (relative to the basis $\XB_{j}$). Understanding the geometry of $\XX$ often requires working simultaneously with several (or even all) rows of $\XX$ and in doing so, we need to count on estimates for the bases $\XB_{j}$). Understanding the geometry of $\XX$ often requires working simultaneously with several (or even all) rows of $\XX$ and in doing so, we need to count on estimates for the bases and spaces and spaces that do not depend on the specific row(s) we are looking at. This compels us to introduce the quantitative versions of the notions we will use and to keep track of the constants involved in our arguments. Finally, Section~\ref{sect:examples} is devoted to applying our theoretical schemes to practical cases. Among the vast amount of novel examples that we can tailor, we exhibit a selection of important new examples of spaces that belong to $\UTAP$ and which involve Lebesgue sequence spaces, Lorentz sequence spaces, Orlicz sequence spaces, Bourgin-Nakano spaces, Hardy spaces, and Tsirelson's space.

\section{Terminology}\label{sect:term}\noindent
We use standard terminology and notation in Banach space theory as can be found, e.g., in \cites{AlbiacKalton2016}. Most of our results, however, will be established in the general setting of quasi-Banach spaces; the unfamiliar reader will find general information about quasi-Banach spaces in \cite{KPR1984}. In keeping with current usage we will write $c_{00}(\Jt)$ for the set of all $(a_j)_{j\in \Jt}\in \FF^{\Jt}$ such that $|\{j\in \Jt \colon a_j\not=0\}|<\infty$, where $\FF$ can be the real or complex scalar field. The convex hull of a subset $Z$ of a vector space will be denoted by $\co(Z)$. A \emph{quasi-norm} on a vector space $\XX$ over $\FF$ is a map $\Vert \cdot\Vert\colon \XX\to[0,\infty)$ satisfying $\Vert x\Vert>0$ when $f\not=0$, $\Vert t\, f\Vert=|t| \, \Vert f\Vert$ for all $t\in\FF$ and all $f\in \XX$, and
\begin{equation}\label{eq:qn}
\Vert f+g\Vert\le \kappa (\Vert f\Vert +\Vert g\Vert), \quad f,\, g\in\XX,
\end{equation}
for some constant $\kappa\ge 1$. The optimal constant such that \eqref{eq:qn} holds will be called the \emph{modulus of concavity} of $\XX$. If $\Vert \cdot\Vert$ verifies
\[
\Vert f+g\Vert^p \le \Vert f\Vert^p+\Vert g\Vert^p, \quad f,\, g\in\XX,
\]
for some $0<p\le 1$, the quasi-norm $\Vert \cdot\Vert$ is said to be a \emph{$p$-norm.} Note that a $p$-norm is a quasi-norm with modulus of concavity at most $2^{1/p-1}$. If $\XX$ is complete with the metric topology induced induced by the quasi-norm, $(\XX,\Vert \cdot\Vert)$ is said to be a \emph{quasi-Banach space}. A \emph{$p$-Banach space} will be a quasi-Banach space equipped with a $p$-norm.
The closed unit ball of a quasi-Banach space $\XX$ will be denoted by $B_\XX$ and
the closed linear span of a subset $Z$ of $\XX$ will be denoted by $[Z]$.

We will frequently index unconditional bases and basic sequences by an unordered countable index set $\Nt$ which need not be the natural numbers $\NN$.
A countable family $\XB=(\xx_n)_{n \in \Nt}$ in $\XX$ is an \emph{unconditional basic sequence} if for every $f\in[\xx_n \colon n \in \Nt]$ there is a unique family $(a_n)_{n \in \Nt}$ in $\FF$ such that the series $\sum_{n \in \Nt} a_n \, \xx_n$ converges unconditionally to $f$.
If $\XB=(\xx_n)_{n \in \Nt}$ is an unconditional basic sequence, there is a constant $K\ge 1$ such that
\[
\left\Vert \sum_{n \in \Nt} a_n \, \xx_n\right\Vert \le K \left\Vert \sum_{n \in \Nt} b_n \, \xx_n\right\Vert
\]
for all finitely non-zero sequence of scalars $(a_n)_{n\in \Nt}$ with $|a_n|\le|b_n|$ for all $n\in\Nt$ (see \cite{AABW2019}*{Theorem 1.10}). If this condition is satisfied some $K\ge 1$ we say that $\XB$ is $K$-unconditional and if, additionally, $[\xx_n \colon n \in \Nt]=\XX$ then $\XB$ is said to be an \emph{unconditional basis} of $\XX$. An unconditional basis $\XB=(\xx_n)_{n \in \Nt}$ in $\XX$ becomes $1$-unconditional under the renorming
\[
\Vert f\Vert_u=\sup\left\{ \left\Vert \sum_{n\in\Nt} a_n \, \xx_n\right\Vert \colon |a_n|\le |\xx_n^*(f)| \right\}, \quad f\in \XX.
\]
Thus, we will in general take the viewpoint that an unconditional basis in a quasi-Banach space $\XX$ confers the structure of an atomic quasi-Banach lattice on $\XX$.

If $\XB=(\xx_n)_{n \in \Nt}$ is an unconditional basis of $\XX$ with biorthogonal functionals $(\xx_n^*)_{n \in \Nt}$, the map $\Fou\colon\XX\to\FF^\Nt$ given by
\[
f=\sum_{n \in \Nt} a_n\, \xx_n \mapsto (\xx_n^*(f))_{n \in \Nt} = (a_n)_{n \in \Nt}
\]
will be called the \emph{coefficient transform} with respect to $\XB$. The \emph{support} of $f\in\XX$ with respect to $\XB$ is the set
\[
\supp(f)=\{n\in\Nt \colon \xx_n^*(f)\not=0\},
\]
and the support of a functional $f^*\in\XX^*$ with respect to $\XB$ is the set
\[
\supp(f^*)=\{n\in\Nt \colon f^*(\xx_n)\not=0\}.
\]

Given a countable set $\Jt$, we write $\EE_{\Jt}:=(\ee_j)_{j\in \Jt}$ for the canonical unit vector system of $\FF^{\Jt}$, i.e., $\ee_j=(\delta_{j,k})_{k\in \Jt}$ for each $j\in \Jt$, where $\delta_{j,k}=1$ if $j=k$ and $\delta_{j,k}=0$ otherwise. A \emph{sequence space} on $\Jt$ will be a quasi-Banach lattice $\LL\subseteq\FF^{\Jt}$ for which the $1$-unconditional basic sequence $\EE_{\Jt}$ is normalized. If $c_{00}$ is dense in $\LL$, so that $\EE_{\Jt}$ is a normalized $1$-unconditional basis of $\LL$, we say that $\LL$ is a \emph{minimal sequence space}. The most important examples of minimal sequence spaces $\LL$ on a set $\Jt$ are the classical Lebesgue sequence spaces $\ell_p(\Jt)$ for $0<p<\infty$, and $c_0(\Jt)$. As is customary, $\ell_p$ will stand for the space $\ell_p(\NN)$ and $\ell_p^s$ will denote $\ell_p(\{n\in\NN \colon n\le s\})$ for $s\in\NN$.

We will refer to a sequence space $\LL$ on $\NN$ as being \emph{subsymmetric} if for each increasing function $\phi\colon\NN\to\NN$, the operator $S_\phi\colon\LL\to\LL$ defined by
\[
(a_n)_{n=1}^\infty\mapsto (b_n)_{n=1}^\infty, \quad \text{where}\;
b_k=\begin{cases} a_n & \text{ if } k=\phi(n),\\ 0 & \text{ otherwise,}\end{cases}
\]
is an isometric embedding. If $S_\phi$ is an isometry for every one-to-one map $\phi$, $\LL$ will be said to be \emph{symmetric}.

Given a sequence space $\LL$ on $\Jt$, and a family $(\XX_j, \Vert\cdot\Vert_{\XX_{j}})_{j\in\Jt}$ of (possibly repeated) quasi-Banach spaces with moduli of concavity uniformly bounded, the space
\[
\left(\bigoplus_{j\in\Jt} \XX_j\right)_\LL=\left\{f=(f_j)_{j\in\Jt}\in\prod_{j\in\Jt} \XX_j\colon \left\Vert (\Vert f_j\Vert_{\XX_{j}})_{n\in\Jt} \right\Vert_\LL<\infty\right\}
\]
is a quasi-Banach space with the quasi-norm
\[
\Vert f\Vert= \left\Vert (\Vert f_j\Vert_{\XX_{j}})_{n\in\Jt}\right\Vert.
\]

Let $(\YY_j)_{j\in\Jt}$ be another collection of (possibly repeated) quasi-Banach spaces. If for each $j\in\Jt$, the map $T_j\colon \XX_j\to\YY_j$ is a bounded linear operator and $M:=\sup_{j\in\Jt} \Vert T_j\Vert<\infty$, then the linear operator
\[
T\colon\left (\bigoplus_{j\in\Jt} \XX_j\right)_\LL\to \left(\bigoplus_{j\in\Jt} \YY_j\right)_\LL,\qquad (f_j)_{j\in\Jt} \mapsto (T_j(f_j))_{j\in\Jt}
\]
is bounded with $\Vert T\Vert\le M$.

The dual space $\LL^*$ of a minimal sequence space on $\Jt$ can be isometrically identified with a sequence space on $\Jt$. Thus, the dual space of $\left(\bigoplus_{j\in\Jt} \XX_j\right)_\LL$ can be isometrically identified with $\left(\bigoplus_{j\in\Jt} \XX_j^*\right)_{\LL^*}$.

For each $k\in\Jt$ let $L_k\colon\XX_k \to (\bigoplus_{j\in\Jt} \XX_j)_\LL$ be the canonical embedding.
If there is a constant $K$ such that, for each $j\in\Jt$, $\XB_j=(\xx_{j,n})_{n\in\Nt_j}$ is a $K$-unconditional basic sequence,
then the sequence
\[
\left(\bigoplus_{j\in\Jt} \XB_j\right)_\LL=\left( L_j(\xx_{j,n})\right)_{n\in\Nt_j,\; j\in\Jt}
\]
is a $K$-unconditional basic sequence of $(\bigoplus_{j\in\Jt} \XX_j)_\LL$. If $\XB_j$ is normalized for all $j\in\Jt$, so is $(\bigoplus_{j\in\Jt} \XB_j)_\LL$. If $\XB_j$ is a basis of $\XX_j$ for all $j\in\Jt$ and $\LL$ is minimal, then $\left(\bigoplus_{j\in\Jt} \XB_j\right)_\LL$ is a basis of $\XX=(\bigoplus_{j\in\Jt} \XX_j)_\LL$ whose dual basis is $\left(\bigoplus_{j\in\Jt} \XB_j^*\right)_{\LL^*}$ via the aforementioned identification between $\XX^*$ and $\left(\bigoplus_{j\in\Jt} \XX_j^*\right)_{\LL^*}$.

If $\Jt$ is finite and $\LL=\ell_\infty(\Jt)$ we set $\bigoplus_{j\in\Jt} \XX_j =(\bigoplus_{j\in\Jt} \XX_j)_\LL$ and $\bigoplus_{j\in\Jt} \XB =(\bigoplus_{j\in\Jt} \XB_j )_\LL$. If $\XX_j=\XX$ for all $j\in\Jt$, we set $\LL(\XX)=(\bigoplus_{j\in\Jt} \XX_j)_\LL$. Similarly, if $\XB_j=\XB$ for all $j\in\Jt$, we set $\LL(\XB) =(\bigoplus_{j\in\Jt} \XX_j)_\LL$. Finally, given $s\in\NN$, we put $\XX^s=\ell_\infty^s(\XX)$ and $\XB^s=\ell_\infty^s(\XB)$.

Suppose that $\XB=(\xx_n)_{n \in \Nt}$ and $\YB=(\yy_n)_{n \in \Nt}$ are families of vectors in quasi-Banach spaces $\XX$ and $\YY$, respectively. Let $C\in(0,\infty)$. We say that $\XB$ $C$-\emph{dominates} $\YB$ if there is a linear map $T$ from $[\XB]$ into $\YY$ with $T(\xx_n)=\yy_n$ for all $n \in \Nt$ and $\Vert T\Vert\le C$. If $T$ is an isomorphic embedding with $\max\{\Vert T\Vert, \Vert T^{-1}\Vert\} \le C\in[1,\infty)$, $\XB$ and $\YB$ are said to be $C$-\emph{equivalent}. We say that $\XB$ is \emph{permutatively $C$-equivalent} to a family $\YB=(\yy_m)_{m\in \Mt}$ in $\YY$, and we write $\XB\sim_C\YB$, if there is a bijection $\pi\colon \Nt\to \Mt$ such that $\XB$ and $(\yy_{\pi(n)})_{n \in \Nt}$ are $C$-equivalent. A \emph{subbasis} of an unconditional basis $(\xx_n)_{n \in \Nt}$ is a family $(\xx_n)_{n\in \Mt}$ for some subset $\Mt$ of $\Nt$.

The symbol $\YB\subsetsim_C \XB$ will mean that the unconditional basic sequence $\YB$ is $C$-equivalent to a permutation of a subbasis of the unconditional basis $\XB$. In all cases, if the precise constants are irrelevant, we simply drop them from the notation.

A sequence $(\xx_{n})_{n\in \Nt}$ in a quasi-Banach space $\XX$ said to be \emph{semi-normalized} if
\[
0<a:=\inf_{n\in\Nt} \Vert \xx_{n}\Vert \le b:=\sup_{n\in\Nt} \Vert \xx_{n}\Vert<\infty.
\]
If $a=b=1$ we say that $(\xx_n)_{n\in\Nt}$ is \emph{normalized}.

Given an unconditional basic sequence $\XB=(\xx_n)_{n\in\Nt}$ and non-zero scalars $(a_n)_{n\in\Nt}$, the rescaled basic sequence $(a_n \, \xx_n)_{n\in\Nt}$ is equivalent to $\XB$ if and only if $(a_n)_{n\in\Nt}$ is semi-normalized. Thus, the properties related to the uniqueness of unconditional bases in quasi-Banach spaces must be stated in terms of normalized (or, equivalently, semi-normalized) basic sequences. We say that a quasi-Banach space $\XX$ has a \emph{unique unconditional basis up to equivalence and permutation} (UTAP unconditional basis for short) if it has a normalized basis $\XB$ and any other normalized basis is permutatively equivalent to $\XB$.

Other more specific terminology will be introduced in context when needed.

\section{Preliminary results}\label{sect:preliminary}
\noindent Our approach to the uniqueness of unconditional basis problem in infinite direct sums of quasi-Banach spaces will rely on an amalgamation of a set of techniques, most of which are specific to the non-locally convex case. In this preparatory section we present the properties and the different methods that will be used in the proofs of our main results in Sections~\ref{sect:SASense} and ~\ref{sect:l1Sense}.

The earliest applications of combinatorial methods to the uniqueness of basis problem can be found in the work of Mitjagin in the early 1970's \cites{Mitja1, Mitja2}, but it was W\'ojtowicz who gave in 1988 a precise formulation of the so-called Schr\"oder-Bernstein principle for unconditional bases (see \cite{Wojtowicz1988}*{Corollary 1}).

\begin{theorem}[Schr\"oder-Bernstein principle for unconditional bases]\label{thm:SBUB}
Let $\XB$ and $\YB$ be unconditional bases of quasi-Banach spaces $\XX$ and $\YY$, respectively. Suppose that $\XB\subsetsim \YB$ and $\YB\subsetsim \XB$. Then $\XB\sim\YB$.
\end{theorem}

Wojtaszczyk rediscovered independently ten years later, in 1997, the idea of using a combinatorial argument in his study of the uniqueness of unconditional basis of $H_{p}(\TT)$ for $0<p<1$ and reproved Theorem~\ref{thm:SBUB} (see \cite{Woj1997}*{Proposition 2.11}). He added to the previous arguments Hall's refinement of the Marriage Lemma (see \cite{Hall1948}*{Theorem 1}).

\begin{theorem}[ Hall's Marriage Lemma]\label{thm:HKL}Let $\Nt$ be a set and $(\Nt_i)_{i\in I}$ be a family of finite subsets of $\Nt$. Suppose that
\[
|F|\le \left| \bigcup_{i\in F} \Nt_i\right|
\]
for all $F\subseteq I$ finite. Then there is a one-to-one map $\phi\colon I\to \Nt$ with $\phi(i)\in \Nt_i$ for every $i\in I$.
\end{theorem}

We next enunciate a simple lemma, whose straightforward proof we omit.
\begin{lemma}\label{lem:equivalenceDirectSums}
Let $\LL$ be a sequence space on a countable set $\Jt$, and for $j\in \Jt$ let $\XX_j$ and $\YY_j$ be quasi-Banach spaces with moduli of concavity uniformly bounded by $\kappa$. Suppose that for each $j\in\Jt$, $\XB_j$ is a normalized $K$-unconditional basic sequence of $\XX_j$ and that $\YB_j$ is an unconditional basic sequence of $\YY_j$ which is $C$-equivalent to $\XB_j$, where $K$ and $C$ are constants independent of $j$.Then the semi-normalized unconditional basic sequence $ (\oplus_{j\in \Jt}\YB_j)_{\LL}$ of $ (\oplus_{j\in \Jt}\YY_j)_{\LL}$ is $C$-equivalent to the normalized unconditional basic sequence $(\bigoplus_{j\in \Jt}\XB_j)_\LL$ of $(\bigoplus_{j\in \Jt}\XX_j)_\LL$.
\end{lemma}

Our first result provides sufficient conditions for an infinite direct sum of unconditional bases to be equivalent to its square.
\begin{lemma}\label{lem:squares}
Let $\LL$ be a sequence space on a countable set $\Jt$. For each $j\in\Jt$ let $\XB_j$ be a normalized $K$-unconditional basis of a quasi-Banach space $\XX_j$ with modulus of concavity bounded above by $\kappa$, where $\kappa$ and $K$ are constants independent of $j$. Suppose that one the the following conditions holds:
\begin{enumerate}[label={{(\alph*)}}, leftmargin=*]
\item\label{it:squaresXj}There is a constant $C$ such that $\XB_j^2\sim_C\XB_j$ for all $j\in\Jt$.
\item\label{it:squaresL} $\LL^2$ is lattice isomorphic to $\LL$, and $\XB_j=\YB$ for all $j\in\Jt$ and some unconditional basis $\YB$.
\item\label{it:SubSym} $\LL$ is subsymmetric, and there is constant $C$ such that, for each $j\in\Jt$, $\XB_j\subsetsim_C\XB_k$ for infinitely many values of $k\in\Jt=\NN$.
\end{enumerate}
Then the basis $\XB=(\bigoplus_{j\in \Jt} \XB_j)_\LL$ is equivalent to a permutation of its square.
\end{lemma}
\begin{proof}
The unconditional basis $\XB^2$ is equivalent to a permutation of $(\bigoplus_{j\in\Jt} \XB_j^2)_\LL$. Thus, if \ref{it:squaresXj} holds, applying Lemma~\ref{lem:equivalenceDirectSums} yields $\XB^2\sim\XB$.

The basis $\XB^2$ is also equivalent to a permutation of $(\bigoplus_{j\in\Jt} \XB_j)_{\LL^2}$.
Therefore, in the cases \ref{it:squaresL} and \ref{it:SubSym}, since $\LL^2$ is lattice isomorphic to $\LL$, $\XB^2$ is equivalent to a permutation of $\XB':=(\bigoplus_{j\in\Nt}\XB_{\phi(j)})_\LL$ for some map $\phi\colon\Nt\to\Nt$. If \ref{it:squaresL} holds, $ \XB_{\phi(j)} =\XB_j$ for all $j\in\Nt$ so that $\XB'=\XB$. Finally, assume that \ref{it:SubSym} holds. Then, we recursively construct an increasing map $\psi\colon\NN\to\NN$ such that
$\XB_{\phi(j)}\subsetsim_C \XB_{\psi(j)}$. By Lemma~\ref{lem:equivalenceDirectSums},
$
\XB'\subsetsim \XB'':=\left( \bigoplus_{n=1}^\infty \XB_{\psi(n)}\right)_\LL
$.
By subsymmetry, $\XB''$ is isometrically equivalent to a subbasis of $\XB$. Hence, by Theorem~\ref{thm:SBUB}, $\XB^2\sim\XB$.
\end{proof}

\subsection{The Cassaza-Kalton Paradigm extended} In a couple of papers of classical elegance (see \cites{CasKal1998, CasKal1999}), Casazza and Kalton crucially used the lattice structure induced by an unconditional basis on a Banach space to provide a much shorter proof than the original one of the uniqueness of unconditional basis UTAP of $c_{0}(\ell_{1})$. Of course, these techniques were not yet available when Bourgain et al.\ wrote their AMS \emph{Memoir} \cite{BCLT1985}, otherwise the proofs of their aforementioned results would have been considerably simpler.

Cassaza and Kalton's methods were transferred to the setting of quasi-Banach lattices and put into practice in \cite{AKL2004} to obtain the uniqueness of unconditional basis UTAP in the spaces $\ell_{1}(\ell_{p})$ and $\ell_{p}(\ell_{1})$ for $0<p<1$, and in \cite{AlbiacLeranozExpoMath} to give a much shorter proof than the original one of the uniqueness of unconditional basis of $\ell_{p}(c_{0})$ for $0<p<1$ (cf.\ \cite{AlbiacLeranoz2002}). The extension of those methods to quasi-Banach lattices requires the notions of $L$-convexity and anti-Euclidean spaces, which we recall next for the convenience of the reader.

A quasi-Banach lattice $\XX$ is said to be \emph{$L$-convex} if there is $0<\varepsilon<1$ so that
\[
\varepsilon \Vert f \Vert \le \max_{1\le i \le k} \Vert f_i\Vert
\]
whenever $f$ and $(f_i)_{i=1}^k$ in $\XX$ satisfy $(1-\varepsilon)kf\ge \sum_{i=1}^k f_i$ and $0\le f_i\le f$ for every $i=1$, \dots, $k$. We say that a family $(\XX_j)_{j\in\Jt}$ of quasi-Banach lattices is $L$-convex if there is $\varepsilon>0$ such each lattice $\XX_j$ is $L$-convex with constant $\varepsilon$ for every $j\in \Jt$. Kalton \cite{Kalton1984b} showed that a quasi-Banach lattice $\XX$ is $L$-convex if and only if it is \emph{$p$-convex} for some $p>0$, that is, for some constant $C$ and all $f_{1},\dots, f_{k}$ in $\XX$ we have
\begin{equation}\label{eq:lconvex}
\left\Vert\left(\sum_{i=1}^k |f_i|^p\right)^{1/p}\right\Vert
\le C \left(\sum_{i=1}^k \Vert f_i\Vert^p\right)^{1/p}.
\end{equation}
The element $(\sum_{i=1}^k \vert f_i\vert^p)^{1/p}$ of $\XX$ is defined via the procedure outlined in \cite{LindenstraussTzafriri1979}*{pp.\ 40-41}. The optimal constant in \eqref{eq:lconvex} will be denoted by $M_p(\XX)$.

Quantitatively, if $\XX$ is $L$-convex with constant $\varepsilon$, there exists $r>0$ and constants $(C_p)_{0<p<r}$ depending only on $\varepsilon$ and the modulus of concavity of $\XX$, such $M_p(\XX)\le C_p$ for all $0<p<r$. Conversely, if $\XX$ is is a $p$-convex quasi-Banach lattice with $M_p(\XX)\le C$, there exist $\kappa$ and $\varepsilon$ depending only on $p$ and $C$ such that $\XX$ is at once an $L$-convex lattice with constant $\varepsilon$ and a quasi-Banach space with modulus of concavity at most $\kappa$. This quantitative approach is perhaps the easiest way to see that if $\LL$ is an $L$-convex sequence space on a set $\Jt$ and $(\XX_j)_{j\in\Jt}$ is a family of $L$-convex quasi-Banach lattices, then $\XX:=(\bigoplus_{j\in\Jt} \XX_j)_\LL$ is an $L$-convex lattice. In fact, if $p>0$ and $C\ge 1$ are such that $M_p(\LL)\le C$ and $M_p(\XX_j)\le C$ for all $j\in\Jt$, then $M_p(\XX)\le C^2$.

A quasi-Banach space $\XX$ is then called \emph{natural} if it is isomorphic to a subspace of an $L$-convex quasi-Banach lattice. Most quasi-Banach spaces arising in analysis are natural. However, it should be pointed out that there are non-natural spaces with an unconditional basis \cite{K86}. It is known \cite{Kalton1984b} that any lattice structure on a natural quasi-Banach space is $L$-convex. Thus, once we make sure that a quasi-Banach space $\XX$ has a lattice structure, the notions of $L$-convexity and naturality become equivalent.

Our results will apply to those natural spaces where the lattice structure is induced by an unconditional basis. In such spaces any unconditional basis induces an $L$-convex lattice structure; then many of the standard techniques of Banach lattice theory can be employed in this setting. For most applications it is easy to verify that the spaces of interest are natural either by by showing that some given unconditional basis is already $p$-convex for some $p>0$ or by identifying them as subspaces of $L$-convex lattices.
\begin{definition}
A family $(\XB_j)_{j\in\Jt}$ of unconditional bases of quasi-Banach spaces $(\XX_j)_{j\in\Jt}$ is said to be $L$-convex if there are constants $K\ge 1$ and $0<\varepsilon<1$ such that $\XB_j$ is $K$-unconditional and it induces an $L$-convex lattice structure on $\XX_j$ with constant $\varepsilon$ for all $j\in\Jt$.
\end{definition}
Notice that if $(\XB_j)_{j\in\Jt}$ is an $L$-convex family of unconditional bases of quasi-Banach spaces $(\XX_j)_{j\in\Jt}$, then the modulus of concavity of the space $\XX_j$ is uniformly bounded. Moreover, if $\LL$ is an $L$-convex sequence space over $\Jt$, then $(\bigoplus_{j\in\Jt} \XB_j)_\LL$ is an unconditional basis of the quasi-Banach space $(\bigoplus_{j\in\Jt} \XX_j)_\LL$ which induces a structure of $L$-convex lattice.

A Banach space $\XX$ is said to be \emph{anti-Euclidean} if it does not contain uniformly complemented copies of finite-dimensional Hilbert spaces. As for $L$-convexity, to deal with families of quasi-Banach spaces we need a more quantitative definition.

\begin{definition}\label{Def:anti-Eucl}
A family $(\XX_j)_{j\in\Jt}$ of Banach spaces is said to be \emph{anti-Euclidean} if for every $R\in(0,\infty)$ there is $k\in\NN$ such that $\Vert S\Vert\, \Vert T\Vert\ge R$ whenever $j\in \Jt$ and $S\colon\ell_2^k\to \XX_j$, $T\colon\XX_j\to\ell_2^k$ are linear operators with $T\circ S=\Id_{\ell_2^k}$.
\end{definition}
By the principle of local reflexivity, a family $(\XX_j)_{j\in\Jt}$ of Banach spaces is anti-Euclidean if and only if $(\XX_j^*)_{j\in\Jt}$ is. The most natural and important examples of anti-Euclidean spaces are $c_0$ and $\ell_1$. Let us bring up a result by Casazza and Kalton.
\begin{theorem}[\cite{CasKal1999}*{Proposition 2.4}]\label{thm:AESums}
Suppose that the countable family $(\XX_j)_{j\in\Jt}$ of Banach spaces is anti-Euclidean. Then the Banach space $(\bigoplus_{j\in\Jt} \XX_j)_{\ell_1}$ is anti-Euclidean.
\end{theorem}

Note that, although Definition~\ref{Def:anti-Eucl} makes sense for quasi-Banach spaces, as a matter of fact we only state it (and will use it) for the ``closest'' Banach spaces to the quasi-Banach spaces we study, i.e., their Banach envelopes. Formally speaking, the \emph{Banach envelope} of a quasi-Banach space $\XX$ consists of a Banach space $\widehat{\XX}$ together with a linear contraction $E_\XX\colon\XX \to \widehat{\XX}$, called the envelope map of $\XX$, satisfying the following universal property: for every Banach space $\YY$ and every linear contraction $T\colon\XX \to\YY$ there is a unique linear contraction $\widehat{T}\colon \widehat{\XX}\to \YY$ such that $\widehat{T}\circ E_\XX=T$. The Banach envelope of a quasi-Banach space can be effectively constructed from the Minkowski functional of $\co(B_\XX)$. This construction shows that $E_\XX(\co(B_\XX))$ is a dense subset of $B_{\widehat{\XX}}$. We say that a Banach space $\YY$ is the Banach envelope of $\XX$ via the map $J\colon\XX\to\YY$ if the associated map $\widehat{J}\colon\widehat{\XX}\to\YY$ is an isomorphism.

The Banach envelope of a minimal sequence space is isometrically isomorphic to a minimal sequence space via the inclusion map (see \cite{AABW2019}*{Proposition 10.9}). We will need the following result.

\begin{proposition}\label{prop:EnvSums}
Let $\LL$ be a minimal sequence space on $\Jt$. Suppose that $\XX_j$ is a quasi-Banach space with modulus of concavity bounded by a uniform constant $\kappa$ for all $j\in\Jt$. Then the Banach envelope of $\XX=(\bigoplus_{j\in\Jt} \XX_j)_\LL$ is isometrically isomorphic to $\YY=(\bigoplus_{j\in\Jt} \widehat{\XX_j})_{\widehat{\LL}}$ via the map
\[
f=(f_j)_{j\in\Jt} \mapsto J(f)=(E_{\XX_j}(f_j))_{j\in\Jt}.
\]
\end{proposition}

\begin{proof}
Since $J$ defines a linear contraction from $\XX$ into $\YY$, it suffices to prove that $J(\co(B_\XX))$ is a dense subset of $B_\YY$. Let $f=(f_j)_{j\in\Jt}\in B_\YY$ and $\varepsilon>0$. For each $j\in\Jt$ set $g_j=f_j/\Vert f_j\Vert$ if $f_j\not=0$ and $g_j=0$ otherwise. Since $g_j\in B_{\widehat{\XX_j}}$, for each $j\in\Jt$ there is $h_j\in \co(B_{\XX_j})$ such that $\Vert g_j-E_{\XX_j}(h_j)\Vert\le\varepsilon/2$. Put $\Gamma=(\gamma_j)_{j\in\Jt}$, where $\gamma_j=\Vert f_j\Vert$ for $j\in\Jt$. Since $\Gamma\in B_{\widehat{\LL}}$, there is $\Lambda\in \co(B_\LL)$ such that $\Vert \Gamma-\Lambda\Vert_{\widehat{\LL}}\le\varepsilon/2$. Moreover, passing to a suitable projection, we can choose $\Lambda$ to be finitely supported. Then, if we denote $\Lambda=(\lambda_j)_{j\in\Jt}$, we have that $h:=(\lambda_j h_j)_{j\in \Jt}\in\co(B_\XX)$. Therefore, if $g=(\lambda_j g_j)_{j\in \Jt}$,
\begin{align*}
\Vert f -E(h)\Vert
&\le \Vert f -g\Vert + \Vert g- E(h)\Vert\\
&=\Vert \Gamma-\Lambda\Vert_{\widehat{\LL}}+\left\Vert ( \lambda_j \Vert g_j -E_j(h_j)\Vert)_{j\in\Jt} \right\Vert_{\widehat{\LL}}\\
&\le \frac{\varepsilon}{2}+\frac{\varepsilon}{2}\Vert \Lambda\Vert_{\widehat{\LL}}\\
&\le \varepsilon.\qedhere
\end{align*}
\end{proof}

In most cases, the proof of the uniqueness of unconditional basis in a given Banach (or quasi-Banach) space also sheds light onto the unconditional structure of its complemented subspaces with an unconditional basis. A sequence $\YB=(\yy_m)_{m\in \Mt}$ in a quasi-Banach space $\XX$ is said to be \emph{complemented} if its closed linear span $\YY= [\YB]$ is a complemented subspace of $\XX$, i.e., there is a bounded linear map $P\colon\XX\to\YY$ with $P|_{\YY}=\Id_\YY$. An unconditional basic sequence $\YB=(\yy_m)_{m\in \Mt}$ is complemented in $\XX$ if and only if there exists a sequence $\YB^*=(\yy_m^*)_{m\in \Mt}$ in $\XX^*$ such that $\yy_m^*(\uu_n)=\delta_{m,n}$ for all $(m,n)\in \Mt^2$ and there is a bounded linear map $P\colon\XX\to \XX$ given by
\begin{equation}\label{eq:projCUBS}
P(f)=P{[\YB,\YB^*]}(f)=\sum_{m\in \Mt} \yy_m^*(f) \, \yy_m, \quad f\in\XX,
\end{equation}
in which case
\[
\Gamma{[\YB,\YB^*]}:=
\sup\left\{ \left\Vert \sum_{m\in M} \yy_m^*(f) \, \yy_m\right\Vert \colon M\subseteq \Mt,\; f\in B_\XX \right\}<\infty.
\]
We will refer to $\YB^*$ as a sequence of \emph{projecting functionals} for $\YB$.

To understand the simplifications derived from taking into account the lattice structure induced by an unconditional basis $\XB$ on the entire space $\XX$, we must look at the supports of $\YB$ and $\YB^*$ with respect to $\XB$.

\begin{definition}Let $\XX$ be a quasi-Banach space with an unconditional basis $\XB$. We say that an unconditional basic sequence $\YB=(\yy_m)_{m\in\Mt}$ is \emph{well complemented} in $\XX$ if it is complemented in $\XX$ and there is a sequence $\YB^*=(\yy_m^*)_{m\in\Mt}$ of projecting functionals for $\YB$ such that:
\begin{enumerate}[label=(\roman*),widest=ii]
\item $\supp(\yy_m^*)\subseteq \supp(\yy_m)$ for all $m\in \Mt$, and
\item $(\supp(\yy_m))_{m\in\Mt}$ is a pairwise disjoint family consisting of finite sets.
\end{enumerate}
In this case, we say that $\YB^*$ is a sequence of \emph{good projecting functionals} for $\YB$. If $\Gamma{[\YB,\YB^*]} \le C$ we will say that $\YB$ is well $C$-complemented and that $\YB^*$ are good $C$-projecting functionals.
\end{definition}

\begin{remark}
Note that a subbasis of a well $C$-complemented basic sequence $(\yy_m)_{m\in\Mt}$ is a a well $C$-complemented basic sequence. In particular, if $(\yy_m^*)_{m\in\Mt}$ are good $C$-projecting functionals, $\Vert \yy_m\Vert\, \Vert \yy_m^*\Vert\le C$ for all $m\in\Mt$.
\end{remark}

The following definition identifies and gives relief to an unstated feature shared by some unconditional bases. Examples of such bases can be found, e.g., in \cites{Kalton1977, CasKal1998, AlbiacLeranoz2008}, where the property naturally arises in connection with the problem of uniqueness of unconditional basis.

\begin{definition}
A normalized unconditional basis $\XB=(\xx_n)_{n\in \Nt}$ of a quasi-Banach space will be said to be \emph{universal for well complemented block basic sequences} if for every normalized well complemented basic sequence $\YB=(\yy_m)_{m\in \Mt}$ of $\XB$ there is a map $\pi\colon \Mt\to \Nt$ such that $\pi(m)\in\supp(\yy_n)$ for every $m\in \Mt$, and $\YB$ is equivalent to the rearranged subbasis $(\xx_{\pi(m)})_{m\in \Mt}$ of $\XB$. In the case when there is a function $\eta\colon[1,\infty)\to [1,\infty)$ such that $\YB$ is $\eta(C)$-equivalent to $(\xx_{\pi(m)})_{m\in \Mt}$ of $\XB$ whenever $\YB$ is well $C$-complemented, we say that $\XB$ is \emph{uniformly universal for well complemented block basic sequences} (with function $\eta$).
\end{definition}

Thus, the following theorem summarizes what can be rightfully called the ``Casazza-Kalton paradigm'' to tackle the uniqueness of unconditional basis problem extended to quasi-Banach lattices. To be able to prove it in this optimal form (even for locally convex spaces) required the very recent solution in the positive of the ``canceling squares'' problem (see \cite{AlbiacAnsorena2020b}).

\begin{theorem}[see \cite{AlbiacAnsorena2020b}*{Theorem 3.9}]\label{thm:paradigm}
Let $\XX$ be a quasi-Banach space with a normalized unconditional basis $\XB$. Suppose that:
\begin{enumerate}[label=(\roman*),leftmargin=*, widest=iii]
\item The lattice structure induced by $\XB$ in $\XX$ is $L$-convex;
\item The Banach envelope of $\XX$ is anti-Euclidean;
\item $\XB$ is universal for well complemented block basic sequences; and
\item $\XB\sim \XB^2$.
\end{enumerate}
Then $\XX$ has a UTAP unconditional basis.
\end{theorem}

\subsection{The peaking property} Another technique that has become crucial to determine the uniqueness of unconditional basis in quasi-Banach spaces is the ``large coefficient technique." It was introduced by Kalton in \cite{Kalton1977} to prove the uniqueness of unconditional basis in nonlocally convex Orlicz sequence spaces $\ell_{F}$. Kalton called a complemented basic sequence $(\yy_{n})$ in $\ell_{F}$ \emph{inessential} if
\[
\inf_{n} \sup_{k} |\yy_n^*(\xx_k)| \, |\xx_k^*(\yy_n)|>0,\]
and proved that if $(\yy_{n})$ is inessential then it is equivalent to the canonical basis $(\xx_{k})$ of $\ell_{F}$.

Kalton's ideas were extended to the general framework of quasi-Banach lattices in \cite{KLW1990}. Here we reformulate this property and regard it as a feature of the unconditional basis $(\xx_n)$ of the space instead of the complemented basic sequence $(\yy_{n})$.

\begin{definition}
An unconditional basis $\XB=(\xx_n)_{n\in \Nt}$ of a quasi-Banach space $\XX$ will be said to have the \emph{peaking property} if for every well complemented basic sequence $\YB=(\yy_m)_{m\in \Mt}$ with respect to $\XB$ there is a sequence $(\yy_m^*)_{m\in \Mt}$ of good projecting functionals such that
\begin{equation}\label{eq:gh}
c:=\inf_{m\in \Mt} \sup_{n\in \Nt} |\yy_m^*(\xx_n)| \, |\xx_n^*(\yy_m)|>0.
\end{equation}
In the case when there is a function $\gamma\colon[1,\infty)\to[1,\infty)$ such that $ \gamma(C)\ge 1/c$ whenever $\YB$ is well $C$-complemented, we say that $\XB$ has the \emph{uniform peaking property} (with function $\gamma$).
\end{definition}

The proof of Proposition~\ref{prop:k2one} below relies on the following reduction lemma which will be used as well in Section~\ref{sect:SASense}.
\begin{lemma}[cf.\ \cite{AlbiacAnsorena2020}*{Lemma 3.1}]\label{lem:k2one}
Let $\YB=(\yy_m)_{m\in\Mt}$ be a well complemented basic sequence with respect to an unconditional basis $\XB=(\xx_n)_{n\in \Nt}$ of a quasi-Banach space $\XX$, and let $(\yy_m^*)_{m\in \Mt}$ be a sequence of good projecting functionals for $\YB$. Suppose $\UB=(\uu_m)_{m\in \Mt}$ and $(\uu_m^*)_{m\in \Mt}$ are sequences in $\XX$ and $\XX^{\ast}$ respectively such that:
\begin{enumerate}[label={{(\roman*)}}, leftmargin=*, widest=iii]
\item $|\xx_n^*(\uu_m)| \le D_1 |\xx_n^*(\yy_m)|$ for all $(n,m)\in \Nt \times \Mt$,
\item $|\uu_m^*(\xx_n)| \le D_2 |\yy_m^*(\xx_n)|$ for all $(n,m)\in \Nt \times \Mt$,
and
\item $|\uu_m^*(\uu_m)|\ge {1}/{D_3}$ for all $m\in \Mt$,
\end{enumerate}
for some positive constants $D_1$, $D_2$ and $D_3$.
Then $\UB$ is a well complemented basic sequence equivalent to $\YB$. Quantitatively, if $\YB$ is well $C$-complemented, and $\XB$ is $K$-unconditional, then:
\begin{enumerate}[label={{(\alph*)}}, leftmargin=*, widest=iii]
\item The sequence $\UB$ is well $B$-complemented with good $B$-projecting functionals $(\lambda_m\, \uu_m^{\ast})_{m\in\Mt}$, where $B=C D_1 D_2 D_3 K^2$ and $\lambda_{m}=1/\uu_m^{\ast}(u_{m})$; and
\item The basic sequence $\YB$ $CD_1K$-dominates $\UB$; and
\item the basic sequence $\UB$ $CD_2D_3K$-dominates $\YB$.
\end{enumerate}
\end{lemma}

\begin{proof}
The proof follows the steps of the proof of \cite{AlbiacAnsorena2020}*{Lemma 3.1}, keeping track of the constants involved.
\end{proof}

\begin{proposition}[cf.\ \cite{AlbiacAnsorena2020b}*{Proposition 3.3}]\label{prop:k2one}
Let $\XX$ be a quasi-Banach space with a normalized unconditional basis $\XB$. If $\XB$ has the peaking property, then $\XB$ is universal for well complemented block basic sequences. Moreover, if $\XB$ is $K$-unconditional and has the uniform peaking property with function $\gamma$, then $\XB$ is uniformly universal for well-complemented block basic sequences with function $C\mapsto K C\gamma(C)$.
\end{proposition}

\begin{proof}
Go through the proof of \cite{AlbiacAnsorena2020b}*{Proposition 3.3} with Lemma~\ref{lem:k2one} in mind, paying attention to the constants involved.
\end{proof}

The following lemma relies on Lemma~\ref{lem:k2one}. Given a family $\XB=(\xx_n)_{n\in \Nt}$ in a quasi-Banach space $\XX$ and $A\subseteq\Nt$ finite, we will use the notation
\[
\Ind_{A}[\XB]=\sum_{n\in A} \xx_n.
\]

\begin{lemma}[cf. \cite{AKL2004}*{Lemma 4.1}]\label{lem:k2two}
Suppose $\XB$ is a normalized $K$-unconditional basis of a quasi-Banach space $\XX$ with dual basis $\XB^*$.
Assume that $\XB$ $D$-dominates the unit vector system of $\ell_1$. If
$B=4C^2DK^2$, then for every normalized $C$-complemented basic sequence $\UB$ in $\XX$ there is a well $B$-complemented basic sequence $\YB=(\yy_m)_{m\in\Mt}$ in $\XX$ such that:
\begin{enumerate}[label=(\roman*), leftmargin=*, widest=iii]
\item $\supp(\yy_m)\subseteq\supp(\uu_m)$ for all $m\in\Mt$;
\item $\YB$ is $E$-equivalent to $\UB$, where $E=2CK\max\{C,D\}$; and
\item $(\Ind_{\supp(\yy_m)}[\XB^*])_{m\in\Mt}$ is a family of good $B$-projecting functionals for $\YB$.
\end{enumerate}
\end{lemma}
\begin{proof}
Just go over the lines of the proof of \cite{AlbiacAnsorena2020b}*{Lemma 3.6} paying attention to the constants involved.
\end{proof}

\subsection{Strongly absolute bases}
Strong absoluteness was identified by Kalton, Ler\'anoz, and Wojtaszczyk in \cite{KLW1990} as the crucial differentiating feature of unconditional bases in quasi-Banach spaces in their investigation of the uniqueness of unconditional bases. One could say that strongly absolute bases are ``purely nonlocally convex'' bases, in the sense that if a quasi-Banach space $\XX$ has a strongly absolute basis, then its unit ball is far from being a convex set and so $\XX$ is far from being a Banach space. Although the term strongly absolute for a basis was coined in \cite{KLW1990}, here we work with a slightly different but equivalent definition.

\begin{definition} An unconditional basis $\XB=(\xx_n)_{n \in \Nt}$ of a quasi-Banach space $\XX$ is \textit{strongly absolute} if for every constant $R>0$ there is a constant $C>0$ such that
\begin{equation}\label{eq:sb}
\sum_{n \in \Nt} |\xx_n^*(f)|\, \Vert \xx_n\Vert \le \max\left\{ C \sup_{n \in \Nt} |\xx_n^*(f)|\,\Vert \xx_n\Vert , \frac{\Vert f \Vert}{R}\right\}, \quad f\in\XX.
\end{equation}
If $\alpha\colon(0,\infty)\to (0,\infty)$ is such that \eqref{eq:sb} holds with $C= \alpha(R)$ for every $0<R<\infty$,
we say that $\XB$ is strongly absolute with function $\alpha$. \end{definition}

Note that if we rescale a strongly absolute basis we obtain a strongly absolute basis with the same function.
Note also that a normalized unconditional basis $\XB=(\xx_n)_{n \in \Nt}$ is strongly absolute with function $\alpha$ if and only if
\[
\Vert f\Vert < R \sum_{n \in \Nt} |\xx_n^*(f)|
\quad \Longrightarrow \quad
\sum_{n \in \Nt} |\xx_n^*(f)| \le \alpha(R) \max_{n\in\Nt} |\xx_n^*(f)|.
\]
If $\XB=(\xx_n)_{n \in \Nt}$ is a strongly absolute basis with function $\alpha$ of a quasi-Banach space $\XX$, the normalized basis $(\xx_n/\Vert \xx_n\Vert)_{n\in\Nt}$ $D$-dominates the unit vector basis of $\ell_1(\Nt)$, where
\[
D=D(\alpha)=\inf_{R>0} \max\left\{ \alpha(R),\frac{1}{R}\right\}.
\]
Roughly speaking a normalized (or semi-normalized) unconditional basis is strongly absolute if and only if it dominates the unit vector basis of $\ell_1$, and whenever the $\ell_1$-norm and the quasi-norm of a vector are comparable then so are the $\ell_{\infty}$-norm and the $\ell_1$-norm of its coordinates.

Adding combinatorial arguments to the methods from \cite{KLW1990} enabled Wojtaszczyk to prove the following criterion for spaces with a strongly absolute basis. Needless to say, he could not count on the Casazza-Kalton paradigm since it had not been discovered yet.

\begin{theorem}[See \cite{Woj1997}*{Theorem 2.12}]\label{WojtCriterion}
Let $\XX$ be a natural quasi-Banach space with a strongly absolute unconditional basis $(\xx_{n})_{n\in \Nt}$. Assume also that $\XX$ is isomorphic to some of its cartesian powers $\XX^{s}$, $s\ge 2$. Then all normalized unconditional bases of $\XX$ are permutatively equivalent.
\end{theorem}

 For further reference, we record an elementary lemma.

\begin{lemma}\label{lem:SADom}
Let $\XB$ and $\YB$ be normalized unconditional bases of quasi-Banach spaces $\XX$ and $\YY$ respectively. Suppose that $\XB$ is strongly absolute with function $\alpha$ and that $\YB$ $D$-dominates $\XB$. Then $\YB$ is strongly absolute with function $D\alpha$.
\end{lemma}

The following proposition guarantees the strongly absoluteness of infinite direct sums of strongly absolute bases. Some applications in Section~\ref{sect:examples} will relay on it as we shall see.

\begin{proposition}\label{NewStronglyAbsBases}
Let $\LL$ be a sequence space on a set $\Jt$ with strongly absolute basis. For each $j\in\Jt$ let $\XB_j$ be a $K$-unconditional basis of a quasi-Banach space with modulus of concavity at most $\kappa$, where $\kappa$ and $K$ are constants independent of $j$. Suppose that there is $\alpha$ such that $\XB_j$ is strongly absolute with function $\alpha$ for all $j\in\Jt$. Then $\XB:=(\bigoplus_{j\in\Nt} \XB_j)_\LL$ is a strongly absolute unconditional basis of $\XX:=(\bigoplus_{j\in\Nt} \XX_j)_\LL$. Moreover, if $\beta$ is a strongly absolute function for the unit vector system of $\LL$, then the map
\[
R\mapsto \gamma(R):=\beta(R D(\alpha)) \alpha( R \beta(R D(\alpha))), \quad 0<R<\infty,
\]
is a strongly absolute function for $\XB$.
\end{proposition}

\begin{proof} Without lost of generality we assume that $\XB_j$ is normalized for all $j\in\Jt$ so that $\XB$ is normalized too. For each $j\in\Jt$ let $\Fou_j$ be the coefficient transform with respect to $\XB_j$. Let $\beta$ be a strongly absolute function for $\LL$.
Pick $f=(f_j)_{j\in \Jt} \in\XX$ and $R\in(0,\infty)$ such that
\[
\Vert f\Vert=\left\Vert (\Vert f_j\Vert)_{j\in\Jt}\right\Vert_\LL\le R
\left\Vert (\Vert \Fou_j(f_j)\Vert_1)_{j\in\Jt}\right\Vert_{1}.
\]
Since, by unconditionality,
\[
\left\Vert (\Vert \Fou_j(f_j)\Vert_1)_{j\in\Jt}\right\Vert_\LL\le D(\alpha)
\left\Vert (\Vert f_j\Vert)_{j\in\Jt}\right\Vert_{\LL},
\]
we obtain
\[
\Vert \Fou(f)\Vert_1=\left\Vert (\Vert \Fou(f_j)\Vert_1)_{j\in\Jt}\right\Vert_{1} \le \beta(R D(\alpha)) \sup_{j\in \Jt} \Vert \Fou_j(f_j)\Vert_1.
\]
Let $k\in\Jt$ be such that $\Vert \Fou_k(f_k)\Vert_1=\sup_{j\in \Jt} \Vert \Fou_j(f_j)\Vert_1$. By unconditionality,
\[
\Vert f_k\Vert\le \Vert f\Vert_\LL\le R \beta(R D(\alpha)) \sup_{j\in \Jt} \Vert \Fou_j(f_j)\Vert_1= R \beta(R D(\alpha)) \Vert \Fou_k(f_k)\Vert_1,
\]
so that,
\[
\sup_{j\in \Jt} \Vert \Fou_j(f_j)\Vert_1 =\Vert \Fou_k(f_k)\Vert_1\le \alpha( R \beta(R D(\alpha))) \Vert \Fou_k(f_k) \Vert_\infty.
\]
Since $\Vert \Fou(f)\Vert_\infty=\sup_j \Vert \Fou_j(f_j) \Vert_\infty$, we obtain
\[
\Vert \Fou(f)\Vert_1\le \gamma(R) \Vert \Fou(f)\Vert_\infty.\qedhere
\]
\end{proof}

\section{Infinite $\LL$-sums of quasi-Banach spaces, where $\LL$ is a sequence space with a strongly absolute basis}\label{sect:SASense}

\noindent Our first theorem in this section uses the previous ingredients and the language introduced in Section~\ref{sect:preliminary} to provide, in particular, an extension of \cite{AKL2004}*{Theorem 4.5}, which established the uniqueness of unconditional basis UTAP in the spaces $\ell_{p}(\ell_{1})$ for $0<p<1$.

\begin{theorem}\label{thm:UWCBSSums}
Let $\LL$ be a sequence space on a set $\Jt$ with a strongly absolute basis. For each $j\in\Jt$, let $\XB_j$ be a normalized $K$-unconditional basis of quasi-Banach space $\XX_j$ with modulus of concavity at most $\kappa$, where $K$ and $\kappa$ are independent of $j$. Suppose that there is a function $\eta\colon[1,\infty)\to[1,\infty)$ such that $\XB_j$ is uniformly universal for well complemented block basic sequences with function $\eta$ for all $j\in\Jt$. Then the unconditional basis $(\bigoplus_{j\in \Jt} \XB_j)_{\LL}$ of the infinite direct sum $\XX:=(\bigoplus_{j\in \Jt} \XX_j)_{\LL}$ is uniformly universal for well complemented block basic sequences.
\end{theorem}

\begin{proof}
We isometrically identify $\XX^*$ with $\VV:=(\oplus_{j\in \Jt} \XX_j^*)_{\LL^*}$ via the natural dual pairing $\langle\cdot,\cdot\rangle\colon\VV\times\XX\to\FF$. For each $j\in \Jt$, let $L_j\colon\XX_j\to \XX$ and $L_j'\colon\XX_j^*\to \VV$ be the natural `inclusion' maps,
and let $T_j\colon \XX\to \XX_j$ be the natural projection. Set $\XB_j=(\xx_{j,n})_{n\in\Nt_j}$. Let $C\in[1,\infty)$ and let $\YB=(\yy_m)_{m\in\Mt}$ be a normalized well $C$-complemented basis sequence with good $C$-projecting functionals $\YB^*=(\yy_m^*)_{m\in\Mt}$. Let $(\vv_m)_{m\in\Mt}$ the corresponding sequence in $\VV$ via the above described dual mapping. Set $\yy_m=(\yy_{j,m})_{j\in\Jt}$ and $\vv_m=(\yy^*_{j,m})_{j\in\Jt}$ for each $m\in\Mt$. Set also
\[
f_m=(\yy^*_{j,m}(\yy_{j,m}))_{\in\Jt}, \quad m\in\Mt.
\]
For $m\in\Mt$, we have
\[
1= \yy_m^*(\yy_m)=| \yy_m^*(\yy_m)|
=\left|\langle \vv_m,\yy_m\rangle\right|=\left|\sum_{j\in\Jt} \yy^*_{j,m}(\yy_{j,m}) \right| \le \Vert f_m\Vert_1
\]
and
\begin{align*}
\Vert f_m\Vert_\LL
&\le \left\Vert (\Vert \yy^*_{j,m}\Vert \, \Vert \yy_{j,m} \Vert)_{j\in\Jt}\right\Vert_\LL\\
&\le \sup_{j\in\Jt} \Vert \yy^*_{j,m}\Vert \left\Vert (\Vert \yy_{j,m} \Vert)_{j\in\Jt}\right\Vert_\LL\\
& =\Vert \vv_m\Vert \, \Vert \yy_m\Vert\\
&\le \Vert \yy_m^*\Vert \, \Vert \yy_m\Vert\\
&\le C.
\end{align*}
Hence, if the unit vector system of $\LL$ is strongly absolute with function $\alpha$, we have
\[
\Vert f_m\Vert_\infty \ge \frac{1}{\alpha(C)}, \quad m\in\Mt.
\]
Therefore, there is a map $\phi\colon\Mt \to \Jt$ such that
\begin{equation*}
\left|\left\langle L'_{\phi(m)}(\yy^*_{\phi(m),m}), L_{\phi(m)}(\yy_{\phi(m),m})\right\rangle\right|
=|\yy^*_{\phi(m),m}(\yy_{\phi(m),m})|\ge \frac{1}{\alpha(C)}
\end{equation*}
for all $m\in\Mt$. By Lemma~\ref{lem:k2one}, the sequence $(L_{\phi(m)}(\yy_{\phi(m),m}) )_{m\in\Mt}$ is $E$-equivalent to $\YB$, where $E=\alpha(C) CK$. We have, in particular,
\[
\frac{1}{E} \Vert \yy_{\phi(m),m}\Vert \le 1, \quad m\in\Mt.
\]
Set $B=E^2CK^2$ and $E'=ECK$.
Applying again Lemma~\ref{lem:k2one} gives that
\[
\UB=(\uu_m)_{m\in\Mt}:=(L_{\phi(m)}(\yy_{\phi(m),m})/\Vert \yy_{\phi(m),m}\Vert )_{m\in\Mt}
\]
is a normalized well $B$-complemented basic sequence $E'$-equivalent to $\YB$. For each $j\in\Jt$ put
\[
\Mt_j=\{ m\in\Mt \colon \phi(m)=j\}.
\]
Composing the projections from $\XX$ onto $\XX$ associated to the well complemented basic sequence $(\uu_m)_{m\in \Mt_j}$ with the maps $L_j$ and $T_j$ we obtain that
\[
\YB_j=(\yy_{j,m}/\Vert \yy_{j,m}\Vert )_{m\in\Mt_j}
\]
is well $B$-complemented in $\XX_j$. By assumption, for each $j\in\Jt$ there is a map $\nu_j\colon\Mt_j\to \Nt_j$ such that $\nu_j(m)\in\supp(\yy_{j,m})$ for all $m\in\Mt_j$ and $\YB_j$ is $\eta(B)$-equivalent to $(\xx_{j,\nu(m)})_{m\in\Mt_j}$. By Lemma~\ref{lem:equivalenceDirectSums}, $\UB$ is $\eta(B)$-equivalent to $(L_{\phi(m)}(\xx_{\phi(m),\nu(m)}))_{m\in\Mt}$.
\end{proof}

We are now ready to obtain the main theoretical result of the section.

\begin{theorem}\label{thm:UTAPLSB}
Let $\LL$ be an $L$-convex sequence space on a countable set $\Jt$. Let $(\XB_j)_{j\in\Jt}$ be an $L$-convex family of normalized unconditional bases of quasi-Banach spaces $(\XX_j)_{j\in\Jt}$. Suppose that:
\begin{enumerate}[label={{(\roman*)}}, leftmargin=*, widest=iii]
\item $\XB_j$ is uniformly universal for well-complemented block basic sequences with function $\eta$ for all $j\in\Jt$;
\item the family of Banach envelopes $\left(\widehat{\XX_j}\right)_{j\in J}$ is anti-Euclidean;
\item The unit vector system of $\LL$ is strongly absolute; and
\item\label{UTAPLSB:SQ} one the the following conditions holds:
\begin{enumerate}[label={{(\alph*)}}, leftmargin=*]
\item\label{UTAPLSB:SQ:1} there is constant $C$ such that $\XB_j^2\sim_C\XB_j$ for all $j\in\Jt$.
\item\label{UTAPLSB:SQ:2} $\LL^2$ is lattice isomorphic to $\LL$, and $\XB_j=\YB$ for all $j\in\Jt$ and some unconditional basis $\YB$.
\item\label{UTAPLSB:SQ:3} $\LL$ is subsymmetric, and there is a constant $C$ such that, for each $j\in\Jt$, $\XB_j\subsetsim_C\XB_k$ for infinitely many values of $k\in\Jt$.
\end{enumerate}
\end{enumerate}
Then $\XX=(\bigoplus_{j\in \Jt} \XX_j)_\LL$ has a UTAP unconditional basis.
\end{theorem}
\begin{proof} Since the unit vector system of $\LL$ is strongly absolute, its Banach envelope is lattice isomorphic to $\ell_1$. By Proposition~\ref{prop:EnvSums}, the Banach envelope of $\XX$ is isomorphic to
$
\left(\bigoplus_{j\in \Jt} \widehat{\XX_j}\right)_{\ell_1},
$
which is anti-Euclidean by Theorem~\ref{thm:AESums}. By Theorem~\ref{thm:UWCBSSums}, $\XB=(\bigoplus_{j\in\Jt} \XB_j)_\LL$ is (uniformly) universal for well complemented block basic sequences. By Lemma~\ref{lem:squares}, $\XB^2\sim\XB$. Applying Theorem~\ref{thm:paradigm} puts an end to the proof.
\end{proof}

\begin{remark}\label{rmk:peakingSA}
A variation of the argument used to prove Theorem~\ref{thm:UWCBSSums} gives that, if we replace the hypothesis ``$\XB_j$ is uniformly universal for well-complemented block basic sequences with function $\eta$ for all $j\in\Jt$'' with
``there is a function $\gamma\colon[1,\infty)\to[1,\infty)$ such that $\XB_j$ has the uniform peaking property with function $\gamma$'', we obtain that the unconditional basis $(\bigoplus_{j\in \Jt} \XB_j)_{\LL}$ of $(\bigoplus_{j\in \Jt} \XX_j)_{\LL}$ has the uniform peaking property. In particular, any strongly absolute unconditional basis $\XB=(\xx_n)_{n\in\Nt}$ of any quasi-Banach space $\XX$ has the uniform peaking property. Let us see a more direct proof of this result. For any $f\in\XX$ and $f^*\in\XX^*$ we have
\[
|f^*(f)|=\left| \sum_{n\in\Nt} f^*(\xx_n)\, \xx_n^*(f) \right|\le \sum_{n\in\Nt}| f^*(\xx_n)|\, |\xx_n^*(f)|,
\]
and, if $\XB$ is $K$-unconditional and normalized,
\[
\left\Vert \sum_{n\in\Nt} | f^*(\xx_n) \, \xx_n^*(f) \, \xx_n \right\Vert
\le
K \sup_{n\in\Nt} |f^*(\xx_n)| \, \Vert f \Vert
\le CK \Vert f^*\Vert \Vert f \Vert.
\]
Hence, if $\Vert f^*\Vert \, \Vert f^*\Vert \le C |f^*(f)|$,
\[
|f^*(f)| \le \alpha(CK) \sup_{n\in\Nt} | f^*(\xx_n)|\, |\xx_n^*(f)|,
\]
where $\alpha$ is the strongly absolute function of $\XB$. We infer that $\XB$ has the uniform peaking property with function $C\mapsto \alpha(CK)$.
\end{remark}

\section{Infinite $\ell_1$-sums of spaces with strongly absolute bases}\label{sect:l1Sense}
\noindent
In this section, we generalize the main result from \cite{AlbiacLeranoz2011} and solve an explicit problem raised ten years ago in \cite{AlbiacLeranoz2011}*{Remark 3.6}. With hindsight, and in light of Theorem~\ref{WojtCriterion}, it also sets right \cite{AlbiacLeranoz2011}*{Corollary 3.4}, whose validity seemed to rely on a wrong set of hypotheses.

\begin{theorem}\label{thm:l1sumK21}
For each $j\in\Jt$, let $\XB_j$ be a normalized $K$-unconditional basis of a quasi-Banach space $\XX_j$ with modulus of concavity at most $\kappa$. Suppose that there is $\alpha$ such that $\XB_j$ is strongly absolute with function $\alpha$ for all $j\in\Jt$. Then the unconditional basis $\XB=(\bigoplus_{j\in \Jt} \XB_j)_{\ell_1}$ of $(\bigoplus_{j\in \Jt} \XX_j)_{\ell_1}$ is uniformly universal for well-complemented block basic sequences.
\end{theorem}

The proof of Theorem~\ref{thm:l1sumK21} will be shortened considerably after taking care of the following lemma.
\begin{lemma}\label{lem:k21aux}
For each $j\in\Jt$, let $\XB_j=(\xx_{j,n})_{n\in \Nt_j}$ be a normalized $K$-unconditional basis of a quasi-Banach space $\XX_j$ with modulus of concavity at most $\kappa$ and let $L_j\colon\XX_j\to\XX:=(\oplus_{j\in \Jt} \XX_j)_{\ell_1}$ be the canonical embedding. For $(j,n)\in\Nt:=\cup_{j\in\Jt} \{j\}\times\Nt_j$ denote $\xs_{j,n}=L_j(\xx_{j,n})$, so that
$\XB:=(\oplus_{j\in \Jt} \XB_j)_{\ell_1}=(\xs_{j,n})_{(j,n)\in\Nt}$.
For each $j\in\Jt$, let $\XB_j^{\ast}=(\xx_{j,n}^*)_{n\in\Nt_j}$ denote the dual basis of $\XB_j$, and let
$\XB^*=(\xs_{j,n}^{\ast})_{(j,n)\in\Nt}$ be the dual basis of $\XB$. Suppose $\YB=(\yy_m)_{m\in\Mt}$ is a normalized well $C$-complemented normalized basic sequence with respect to the normalized unconditional basis $\XB$ of $\XX$ for which $
(\yy_m^*)_{m\in\Mt} = (\Ind_{\supp(\yy_m)}[\XB^*]) _{m\in \Mt}
$
is a family of good $C$-projecting functionals. Put $\yy_m=(\yy_{j,m})_{j\in \Jt}$ and set
\[
J_m=\{j\in \Jt \colon \yy_{j,m}\not=0\}.
\]
Then:
\begin{enumerate}[label=(\alph*), leftmargin=*]
\item\label{lem:k21aux:z}
If $\XB_j$ $D$-dominates the unit vector system of $\ell_1$ for all $j\in\Jt$, $(\yy_m)_{m\in\Mt}$ $D$-dominates the unit vector system of $\ell_1(\Mt)$.

\item\label{lem:k21aux:a}
If $|M|\le |\cup_{m\in M} J_m|$ for every $M\subseteq \Mt$ finite, there is a one-to-one map $\pi\colon \Mt \to \Nt$ such that
the rearranged subbasis $(\xs_{\pi(m)})_{m\in \Mt}$ of $\XB$ is isometrically equivalent to the unit vector system of $\ell_1$ and $C$-dominates $\YB$.

\item\label{lem:k21aux:b}
If
\begin{enumerate}[label=(\roman*), leftmargin=*, widest=iii]
\item $\XB_j$ is strongly absolute with function $\alpha$ for every $j\in\Jt$, and
\item there is $M\subseteq \Mt$ finite and nonempty such that $|\cup_{m\in M} J_m| < |M|$,
\end{enumerate}
for every $R\in(0,\infty)$ we have
\begin{equation*}
\Delta:=\sup\left\{ |\xx_{j,n}^*(\yy_{j,m})| \colon m \in \Mt,\, j\in \Jt, \, n\in \Nt_j\right\} \ge \frac{R-C}{R\, \alpha(R)}.
\end{equation*}
\end{enumerate}
\end{lemma}

\begin{proof} \ref{lem:k21aux:z}
The basis $\XB$ $D$-dominates the unit vector system of $\ell_1(\Nt)$. That is,
\[
\sum_{(j,n)\in\Nt}| \xs^*_{j,n}(f) |\le D \Vert f \Vert, \quad f\in\XX.
\]
For $(a_m)_{m\in \Mt}\in c_{00}(\Mt)$, write $f=\sum_{m\in\Mt} a_m\, \yy_m$. Then,
\begin{align*}
\sum_{m\in\Mt} |a_m|
&=\sum_{m\in\Mt} \left|\yy_{m}^*\left( f \right)\right| \\
&=\sum_{m\in\Mt}\sum_{(j,n)\in\supp(\yy_{m})} \left| \xs_{j,n}^*(f)\right| \\
&=\sum_{(j,n)\in\Nt} \left| \xs_{j,n}^*(f)\right| \\
&\le D\Vert f\Vert.\qedhere
\end{align*}

\ref{lem:k21aux:a}
By Theorem~\ref{thm:HKL}, there is a one-to-one map $\phi\colon \Mt \to \Jt$ such that $\yy_{\phi(m) ,m}\not=0$ for all $m\in\Mt$. Thus, there is $\nu\colon\Mt \to \cup_{j\in \Jt} \Nt_j$ such that $\nu(m)\in \Nt_{\phi(m)}$ for all $m\in\Mt$ and such that $\xs^*_{\phi(m),\nu(m)}(\yy_m)\not=0$. Define $\pi\colon\Mt\to\Nt$ by $\pi(m)=(\phi(m),\nu(m))$ for all $m\in\Mt$. Let $T_j$ be the canonical projection of $\XX$ onto $\XX_j$. Since, given $m\in\Mt$, $T_j(\xs_{\pi(m)})\not=0$ for at most one $j\in\Jt$, for every $(a_m)_{m\in\Mt}\in c_{00}(\Mt)$ we have
\begin{align*}\label{eq:k21aux2}
\left\Vert\sum_{m\in \Mt} a_m\, \xs_{\pi(m)} \right\Vert
&=\sum_{j\in\Jt} \left\Vert\sum_{m\in \Mt} a_m\, T_j( \xs_{\pi(m)}) \right\Vert \\
&=\sum_{j\in\Jt} \sum_{m\in \Mt} |a_m| \, \Vert T_j( \xs_{\pi(m)}) \Vert \\
&= \sum_{m\in \Mt} |a_m| \sum_{j\in\Jt} \Vert T_j( \xs_{\pi(m)}) \Vert \\
&= \sum_{m\in \Mt} |a_m| \, \Vert \xx_{\pi(m)}) \Vert \\
&= \sum_{m\in \Mt} |a_m|.
\end{align*}

Let $P=P[\YB,\YB^*]$ be the projection defined in \eqref{eq:projCUBS}. If $(j,n)\in\sup(\yy_m)$ for some $m\in\Nt$,
\[
P(\xs_{j,n})=\sum_{m'\in\Mt} \yy_{m'}^*(\xs_{j,n}) \, \yy_{m'}=\sum_{m'\in\Mt} \delta_{m,m'} \, \yy_{m'}=\yy_m.
\]
Hence, if $f=\sum_{m\in\Mt} a_m\, \yy_m$,
\begin{equation*}
\left\Vert f \right\Vert
=\left\Vert P\left( \sum_{m\in \Mt} a_m\, \xs_{\pi(m)}\right)\right\Vert
\le C \left\Vert \sum_{m\in \Mt} a_m\, \xs_{\pi(m)} \right\Vert.\qedhere
\end{equation*}

\ref{lem:k21aux:b}
Note that $\XB_j$ $D(\alpha)$-dominates the unit vector system of $\ell_1(\Nt_j)$ for all $j\in\Jt$.
Pick $\Mt_0$ minimal with $|\cup_{m\in \Mt_0} J_m|<|\Mt_0|<\infty$. Since $J_m\not=\emptyset$ for all $m\in \Mt$ we have $|\Mt_0|\ge 2$. Pick $m_0\in\Mt_0$ arbitrary and set $M=\Mt_0\setminus\{m_0\}$. By Lemma~\ref{lem:k21aux}~\ref{lem:k21aux:a}, the unit vector system of $\ell_1(M)$ $C$-dominates the finite basis $(\yy_m)_{m\in M}$. If we set $J=\cup_{m\in M} J_m$,
\begin{align*}
|M| & = \left| \sum_{m\in M} \sum_{(j,n)\in\Nt} \xs_{j,n}^*(\yy_m)\right|\\
&= \left| \sum_{j\in J} \sum_{n\in\Nt_j} \xx_{j,n}^*\left(\sum_{m\in M} \yy_{j,m}\right)\right|\\
&\le \sum_{j\in J} \sum_{n\in\Nt_j} \left| \xx_{j,n}^*\left(\sum_{m\in M} \yy_{j,m}\right)\right|\\
&\le \sum_{j\in J} \max\left\{ \alpha(R) \sup_{n\in\Nt_j} \left|\xx_{j,n}^*\left(\sum_{m\in M} \yy_{j,m}\right)\right|, \frac{1}{R} \left\Vert\sum_{m\in M} \yy_{j,m}\right\Vert \right\}\\
&\le \sum_{j\in J} \max\left\{ \alpha(R) \Delta, \frac{1}{R} \left\Vert\sum_{m\in M} \yy_{j,m}\right\Vert \right\}\\
&\le |J| \alpha(R) \Delta + \frac{1}{R} \sum_{j\in J} \left\Vert\sum_{m\in M} \yy_{j,m}\right\Vert \\
&= |J| \alpha(R) \Delta + \frac{1}{R} \left\Vert\sum_{m\in M} \yy_{m}\right\Vert\\
&\le |J| \alpha(R) \Delta + \frac{C}{R} |M|
\end{align*}
Since
\[
|J|\le |\cup_{m\in \Mt_0} J_m|\le |\Mt_0|-1=|M|,\]
we are done.
\end{proof}

\begin{proof}[Proof of Theorem~\ref{thm:l1sumK21}.]
Let $C\in[1,\infty)$. Pick $R> B:=4C^2D(\alpha) K^2$, $E>R \alpha(R)/(R-C)$ and $E'=2CK\max\{C,D(\alpha)\}$. Let $\UB$ be a well $C$-complemented basic sequence in $\XB$. By Lemma~\ref{lem:k2two}, there is a well $B$-complemented basic sequence $\YB=(\yy_m)_{m\in\Mt}$ in $\XB$ with good $B$-projecting functionals
\[
(\Ind_{\supp(\yy_m)}[\XB^*])_{m\in\Mt}
\]
which is $E'$-equivalent to $\UB$. With the terminology of Lemma~\ref{lem:k21aux}, put
\begin{align*}
\Mt_0&=\{ m\in\Mt \colon |\xx_{j,n}^*(\yy_{j,m})| \le \frac{1}{E} \text{ for all } (j,n)\in\Nt \} \text{ and } \\
\Mt_1&=\Mt\setminus\Mt_0.
\end{align*}
By Lemma~\ref{lem:k21aux} there is $\pi_0\colon\Mt_0 \to \Nt$ such that $\pi_0(m)\in\supp(\yy_m)$ for all $m\in\Mt_0$ and
$(\yy_m)_{m\in\Mt_0}$ $D(\alpha)$-dominates and it is $B$-dominated by $(\xs_{\pi_0(m)})_{m\in\Mt_0}$. In turn, there is $\pi_1\colon\Mt_1\to\Nt$ such that
\[
|\xs^*_{\pi_1(m)}(\yy_m)|>\frac{1}{E}, \quad m\in\Mt_1.
\]
Hence, by Lemma~\ref{lem:k2one}, $(\yy_m)_{m\in\Mt_1}$ $BK$-dominates and it is $BKE$-dominated by $(\xs_{\pi_1(m)})_{m\in\Mt_1}$. We infer that if
\begin{align*}
D_1&=\kappa C \max\{ BK, D(\alpha)\}=4\kappa C^3 D(\alpha) K^2,\\
D_2&=\kappa K \max\{ BKE, B\}= 4 \kappa C^2D(\alpha) K^4 E,
\end{align*}
and $\pi\colon\Mt\to\Nt$ is obtained by glueing the functions $\pi_0$ and $\pi_1$, $\UB$ $D_1$-dominates and it is $D_2$-dominated by $(\xs_{\pi(m)})_{m\in\Mt}$.
\end{proof}

\begin{theorem}\label{thm:l1SAUTAP}
Let $(\XB_j)_{j\in\Jt}$ be an $L$-convex family of normalized unconditional bases of quasi-Banach spaces $(\XX_j)_{j\in\Jt}$. Suppose that:
\begin{enumerate}[label=(\alph*), leftmargin=*, widest=a]
\item $\XB_j$ is strongly absolute with the same function $\alpha$ for all $j\in\Jt$; and

\item\label{l1SAUTAP:SQ} Either:

\begin{enumerate}[label=(\roman*), leftmargin=*, widest=ii]
\item\label{l1SAUTAP:SQ:1} There is a constant $C$ such that $\XB_j^2\sim_C\XB_j$ for all $j\in\Jt$, or

\item\label{l1SAUTAP:SQ:2} There is a constant $C$ such that, for each $j\in\Jt$, $\XB_j\subsetsim_C\XB_k$ for infinitely many values of $k\in\Nt$.
\end{enumerate}
\end{enumerate}
Then the space $\XX=(\bigoplus_{j\in \Jt} \XX_j)_{\ell_1}$ has a UTAP unconditional basis.
\end{theorem}

\begin{proof}
Since $\XB_j$ $D(\alpha)$-dominates the unit vector system of $\ell_1(\Nt_j)$ for all $j\in\Jt$, the normalized basis $\XB=(\bigoplus_{j\in \Jt} \XB_j)_{\ell_1}$ of $\XX$ $D(\alpha)$-dominates the unit vector system of $\ell_1(\Nt)$. Hence the Banach envelope of $\XX$ is isomorphic to the anti-Euclidean space $\ell_1(\Nt)$. By Theorem~\ref{thm:l1sumK21} and Remark~\ref{rmk:peakingSA}, $\XB$ is (uniformly) universal for well complemented block basic sequences. By Lemma~\ref{lem:squares}, $\XB^2\sim\XB$ and so applying Theorem~\ref{thm:paradigm} concludes the proof.
\end{proof}

\section{Applications and Examples}\label{sect:examples}

\noindent Theorems~\ref{thm:UTAPLSB} and \ref{thm:l1SAUTAP}, combined with Proposition~\ref{NewStronglyAbsBases}, yield a myriad of new examples of quasi-Banach spaces with a UTAP unconditional bases. In this section we highlight applications only to a sampler of infinite direct sums involving classical spaces, but the reader is encouraged to create their favourite infinite direct sums and use our previous results to check that they enjoy the property of uniqueness (UTAP) of unconditional basis. The possibilities for new examples are endless.

\subsection{Lorentz Sequence Spaces}

Let $\ww=(w_n)_{n=1}^\infty$ be a sequence of non-negative scalars with $w_1>0$ and $(s_n)_{n=1}^\infty$ be the \emph{primitive weight} of $\ww$, defined by
\[
s_n=\sum_{k=1}^n w_k,\quad n\in \NN.
\]
Given $0<p<\infty$ and $0<q\le \infty$, the \emph{Lorentz sequence space} $d_{p,q}(\ww)$ consists of all $f\in c_0$ whose non-increasing rearrangement $(a_n)_{n=1}^\infty$ satisfies
\[
\Vert f\Vert_{p,q,\ww}:=\left( \sum_{n=1}^\infty ( a_n s_n^{1/p})^q \frac{w_n}{s_n}\right)^{1/q} <\infty,
\]
with the usual modification if $q=\infty$. If $(s_n)_{n=1}^\infty$ is \emph{doubling}, i.e., $\sup_m s_{m}/s_{\lceil m/2\rceil}<\infty$, then $\Vert \cdot\Vert_{p,q,\ww}$ is a quasi-norm. In this case, $d_{p,q}(\ww)$ is a symmetric sequence space. Moreover, if $q<\infty$, $d_{p,q}(\ww)$ is minimal.

If $0<p=q<1$, the space $d_{p,p}(\ww)$ coincides with the Lorentz space denoted as $d(\ww, p)$ by Altshuler in \cite{Alt1975} (see also \cites{CRS2007,AlbiacLeranoz2008,AlbiacAnsorena2020b}), while $d_{p,\infty}(\ww)$ coincides with the weak Lorentz space denoted as $d^\infty(\ww, p)$ in \cite{CRS2007}. If $q=\infty$, we will denote by $\ls_{p,\infty}(\ww)$ the separable part of $d_{p,\infty}(\ww)$, i.e., the closed linear span of $c_{00}$ in $d_{p,\infty}(\ww)$.

If $0<q<r\le \infty$, we have
\begin{equation}\label{eq:LorentzEmb}
d_{p,q}(\ww) \subseteq d_{p,r}(\ww),
\end{equation}
and for all $A\subset \NN$ with $|A|=m$,
\begin{equation}\label{eq:LorentzFundamental}
\Vert \Ind_A \Vert_{p,q,\ww} \approx s_m^{1/p}.
\end{equation}
Thus it could be said that for fixed $p$ and $\ww$, the spaces $d_{p,q}(\ww)$ are close to each other in the sense that all of them share (essentially) the fundamental function of the canonical basis. This is important to be taken into account when considering embeddings (see below).

We point out that if $0<p,q<\infty$ and the primitive weight of $\ww'$ is $(s_n^{p/q})_{n=1}^\infty$, then
\begin{equation}\label{eq:rescale}
d_{p,q}(\ww)= d_{q,q}(\ww^{\prime}),
\end{equation}
up to an equivalent norm. Similarly, if $\ww'=(w_n')_{n=1}^\infty$ denotes the weight whose primitive weight is $(s_n^{1/p})_{n=1}^\infty$,
\[
d_{p,\infty}(\ww)=d_{1,\infty}(\ww').
\]
Thus, every sequence Lorentz space $d_{p,q}(\ww)$ can be identified, up to an equivalent quasi-norm, with a Lorentz sequence space $d_{1,q}(\ww')$ for a suitable weight $\ww'$. The advantages of establishing results concerning sequence Lorentz spaces in terms of the scale of spaces $d_{1,q}(\ww)$, $0<q\le\infty$, must be understood in light of \eqref{eq:LorentzEmb} and \eqref{eq:LorentzFundamental}. For a concise introduction to Lorentz sequence spaces we refer to \cite{AABW2019}*{\S9.2}. Next we focus on the features of these spaces that are of interest for the purposes of this paper.

We next include a proof of the fact that all Lorentz sequence spaces are $L$-convex. To that end we need to introduce the following concept.

We say that a sequence $(\Phi(m))_{m=1}^\infty$ of positive scalars has the \emph{upper regularity property} (URP for short) if there is $r\in\NN$ such that
\[
\Phi(rn)\le \frac{r}{2} \Phi(n), \quad n\in\NN.
\]
If $(\Phi(m))_{m=1}^\infty$ has the URP, then there are $0<\alpha<1$ and $0<C<\infty$ such that
\[
\frac{\Phi(n)}{\Phi(m)} \le C \frac{n^\alpha}{m^{\alpha}}, \quad m\le n
\]
(see \cite{DKKT2003}*{\S4}). This implies that
\begin{equation}\label{eq:URP:Regular}
\sum_{n=1}^m \frac{1}{\Phi(n)} \le C' \frac{m}{\Phi(m)}, \quad m\in\NN.
\end{equation}
for some constant $C'$.

Given $r\in(0,\infty)$, the \emph{$r$-convexification} of a sequence space $\LL$ on a countable set $\Jt$ is the sequence space consisting of all $f\colon\Jt\to\FF$ such that $|f|^r\in \LL$. By definition, $\LL$ is lattice $p$-convex if and only its $r$-convexification is $(pr)$-convex. Note that the $r$-convexification of $d_{p,q}(\ww)$ is $d_{pr,qr}(\ww)$.

\begin{theorem}[cf.\ \cite{KaminskaParrish2008}*{Theorem 8} and \cite{Jameson1998}*{Proposition 1}]
Let $\ww$ be a weight whose primitive weight $(s_n)_{n=1}^\infty$ is doubling. Then $d_{1,q}(\ww)$ is $L$-convex for all $0< q \le\infty$. \end{theorem}
\begin{proof} We will show that $d_{1,q}(\ww)$ is $r$-convex for some $r>0$.
Let $C\in[1,\infty)$ be such that $s_{2n} \le C s_n$ for all $n\in\NN$. Pick $\alpha_0\in(0,1)$ and $k\in\NN$ such that
$
C^{\alpha_0} \le 2^{1-1/k}
$.
Then, if $r=2^k$,
\[
s_{rn}^\alpha \le \frac{1}{2} r s_n^\alpha, \quad n\in\NN,\; 0<\alpha<\alpha_0.
\]
That is, $(s_n^\alpha)_{n=1}^\infty$ has the URP and, then, satisfies inequality \eqref{eq:URP:Regular} for all $0<\alpha<\alpha_0$. Let $\ww_\alpha$ be the weight whose primitive weight if $(s_n^\alpha)_{n=1}^\infty$. By \cite{CRS2007}*{Theorems 2.5.10 and 2.5.11}, $d_{p,q}(\ww_{\alpha p})$ is locally convex for all $\alpha<\alpha_0$ and $q>1$. Note that local convexity is equivalent to lattice $1$-convexity. We infer that $d_{1,q}(\ww)$ is lattice $r$-convex for all $0<r<\min\{\alpha_0,q\}$.
\end{proof}

Next we tackle the strong absoluteness or the canonical basis of Lorentz sequence spaces.

\begin{proposition}\label{prop:LorentzSA}
Suppose that the primitive weight $(s_n)_{n=1}^\infty$ of $\ww=(w_n)_{n=1}^\infty$ is doubling.
\begin{enumerate}[label=(\alph*), leftmargin=*]
\item\label{prop:LorentzSA:W} The following are equivalent:
\begin{enumerate}[label=(\roman*), leftmargin=*, widest=iii]
\item\label{it:WL1} $d_{1,\infty}(\ww)$ is continuously included in $\ell_1$.
\item\label{it:WL2} $\sum_{n=1}^\infty 1/s_n <\infty$.
\item\label{it:WL3} The unit vector system is a strongly absolute basis of $\ls_{1,\infty}(\ww)$.
\end{enumerate}
\item\label{prop:LorentzSA:NW}
Let $1<q<\infty$ and let $q'$ be its conjugate exponent. Suppose that $\sum_{n=1}^\infty w_n^{-q'+1} s_n^{-1} <\infty$. Then the unit vector system is a strongly absolute basis of $d_{1,q}(\ww)$.
\item\label{prop:LorentzSA:NWO}
Let $0<q\le 1$.
\begin{enumerate}[label=(\roman*), leftmargin=*, widest=ii]
\item\label{it:Lq:1} $d_{1,q}(\ww)$ is continuously included in $\ell_1$ if and only if $\inf_n s_n/n>0$. Moreover.
\item\label{it:Lq:2} if $\lim_n s_n/n=\infty$, then the unit vector system is a strongly absolute basis of $d_{1,q}(\ww)$.
\item\label{it:Lq:3} if $\inf_n s_n/n>0$ and $q<1$, the unit vector system of $d_{1,q}(\ww)$ is uniformly universal for well complemented block basic sequences.
\end{enumerate}

\end{enumerate}
\end{proposition}

\begin{proof}
The implication \ref{it:WL3} $\Rightarrow$ \ref{it:WL1} in \ref{prop:LorentzSA:W} is obvious. If $f=(1/s_n)_{n=1}^\infty$, we have $\Vert f \Vert_{1,\infty,\ww}=1$. This yields \ref{it:WL1} $\Rightarrow$ \ref{it:WL2}. To prove \ref{prop:LorentzSA:NW} and the implication \ref{it:WL2} $\Rightarrow$ \ref{it:WL3} in \ref{prop:LorentzSA:W} we pick $1<q\le \infty$ and $0<R<\infty$. Choose $m=m(R)\in\NN$ such that
\[
\sum_{n=m+1}^\infty =\frac{1}{w_n^{q'}} \frac{w_n}{s_n}\le \frac{1}{(2R)^{q'}}.
\]
Let $f\in\FF^\NN$ and denote by $(a_n)_{n=1}^\infty$ its non-increasing rearrangement.
By Holder's inequality,
\begin{align*}
\Vert f\Vert_1&
=\sum_{n=1}^m a_n + \sum_{n=m+1}^\infty \frac{1}{w_n} a_n s_n \frac{w_n}{s_n}\\
& \le m \Vert f\Vert_\infty+\frac{1}{2R} \left( \sum_{n=m+1}^\infty a_n^q s_n^q \frac{w_n}{s_n}\right)^{1/q} \\
&\le m \Vert f\Vert_\infty+\frac{1}{2R} \Vert f\Vert_{1,q,\ww}.
\end{align*}
Thus, if $\Vert f\Vert_{1,q,\ww}\le R \Vert f\Vert_1$, we obtain $\Vert f\Vert_1\le 2m a_1=2m\Vert f\Vert_\infty$.

As far as \ref{prop:LorentzSA:NWO} is concerned, the ``only if'' part in \ref{it:Lq:1} is clear. By \eqref{eq:LorentzEmb}, to prove the converse it
suffices to consider the case $q=1$. If $\ww'=(w'_{n})_{n=1}^\infty$ is the weight defined by $w'_n=1$ for all $n\in\NN$, then $d_{1,1}(\ww)\subseteq d_{1,1}(\ww')$. Since $d_{1,1}(\ww')=\ell_1$ we are done.

\ref{it:Lq:2} is essentially known (see \cite{NawOrt1985}*{Lemma 4} and \cite{KLW1990}*{Theorem 2.6}). However, as an explicit proof is not available in the literature, we next include one for the sake of completeness.
Again, by Lemma~\ref{lem:SADom}, it suffices to consider the case $q=1$.
Let $R\in(0,\infty)$. Choose $m\in\NN$ such that $s_n\ge 2R n$ for all $n\ge m+1$.
If $(a_n)_{n=1}^\infty$ is the non-increasing rearrangement of $f$, by Abel's summation formula,
\begin{align*}
\sum_{n=1}^\infty a_n &
=\sum_{n=1}^\infty (a_n-a_{n+1}) n \\
&\le \sum_{n=1}^m (a_n-a_{n+1}) n + \frac{1}{2R} \sum_{n=m+1}^\infty (a_n -a_{n+1}) s_n\\
&\le -m a_{m+1}+\sum_{n=1}^m a_n + \frac{1}{2R} \sum_{n=1}^\infty (a_n -a_{n+1}) s_n\\
&\le m \Vert f \Vert_\infty+\frac{1}{2R} \Vert f\Vert_{1,1,\ww}.
\end{align*}
Therefore, $\Vert f\Vert_1 \le 2m$ whenever $ \Vert f\Vert_{1,1,\ww}\le R \Vert f\Vert_1$.

In regards to \ref{it:Lq:3}, we point out that it was proved in \cite{AlbiacAnsorena2020b}*{Proposition4.2} that $d_{1,q}(\ww)$ has the peaking property. A close look at the proof of this result reveals that, in fact, it has the uniform peaking property. Essentially, this is due to the validity of a constructive version of the proof of \cite{Popa1981}*{Lemma 3.1}. Specifically, there is a function $\zeta\colon(0,\infty)\to (0,\infty)$ depending on $p$ and $\ww$ such that every normalized disjointly supported sequence $(\yy_m)_{m=1}^\infty$ with respect to $(\ee_{n})_{n=1}^{\infty}$ with $\liminf_m \Vert \yy_m \Vert_\infty< \zeta(\varepsilon)$ has a subsequence that $(1+\varepsilon)$-dominates the unit vector system of $\ell_p$. By Proposition~\ref{prop:k2one}, $(\ee_{n})_{n=1}^{\infty}$ is uniformly universal for block basic sequences.
\end{proof}

To complement the theoretical contents of this section we shall introduce lattice concavity and a quantitative tool from approximation theory that serves in particular to measure how far an unconditional basis is from the canonical $\ell_{1}$-basis. The main idea is to use embeddings into Lorentz sequence spaces to deduce that certain bases are strongly absolute.

Given a (semi-normalized)
unconditional basis $\XB$ of a quasi-Banach space $\XX$ we define its \emph{lower democracy} function as
\[
\varphi_l[\XB](m)=\inf_{|A|\ge m} \Vert \Ind_A[\XB]\Vert, \quad m\in\NN.
\]
If $\LL$ is a sequence space, $\varphi_l[\LL]$ will denote the lower democracy function of its unit vector system.
The quasi-Banach lattice $\LL$ is said to be \emph{$q$-concave}, $0<q\le \infty$, if there is a nonnegative constant $C$ such that
\[
\left(\sum_{i=1}^k \Vert f_i\Vert^q\right)^{1/q}
\le C \left\Vert\left(\sum_{i=1}^k |f_i|^q\right)^{1/q}\right\Vert, \quad f_i\in \LL.
\]
Any quasi-Banach lattice is trivially $\infty$-concave.

\begin{theorem}\label{thm:embLorentz}
Let $\XX$ be a quasi-Banach space with a normalized unconditional basis $\XB=(\xx_n)_{n=1}^\infty$. Suppose that $\XB$ induces a $q$-concave lattice structure on $\XX$ for some $0<q\le \infty$. Let $\ww=(w_n)_{n=1}^\infty$ be a weight with
\[
s_m:=\sum_{n=1}^m w_n \le \varphi_l[\XB](m), \quad m\in\NN.
\]
Then $\XB$ dominates the unit vector system of $d_{1,q}(\ww)$, i.e., $\XX$ continuously embeds into $d_{1,q}(\ww)$ via $\XB$.
\end{theorem}

\begin{proof}
If $q=\infty$ the result is known (see \cite{AADK2016}*{Lemma 6.1}). Suppose that $q<\infty$.
Put $\ww'=(w_n')_{n=1}^\infty$, where $w'_n=s_n^q-s_{n-1}^q$ with the convention that $s_0=0$.
Let $(a_n)_{n=1}^\infty\in c_{00}$ such that $(|a_n|)_{n=1}^\infty$ is non-increasing. Put $t=|a_1|$ and for each $k\in\NN$ consider the set
\[
J_k=\{n\in\NN \colon t 2^{-k}< |a_n| \le t 2^{-k+1}\}.
\]
Notice that $(J_k)_{k=1}^\infty$ is a partition of $\{n\in\NN \colon a_n\not=0\}$. Set $m_0=0$ and for $j\in\NN$ put $m_j=\sum_{k=1}^j |J_k|$, so that $J_j=\{n\in\NN \colon m_{j-1}+1\le n \le m_j\}$ for all $j\in\NN$.
Define
\[
f_k=2^k \sum_{n\in J_k} a_n\,\xx_n.
\]
By Abel's summation formula,
\[
f:=\sum_{n=1}^\infty a_n \, \xx_n=\sum_{k=1}^\infty 2^{-k} f_k =\frac{1}{2}\sum_{k=1}^\infty 2^{-j} \sum_{k=1}^j f_k.
\]
Therefore, if $C$ is the $q$-concavity constant of $\XB$ and $K$ is its unconditionality basis constant,
\begin{align*}
\Vert f \Vert^q& \ge \frac{t^q}{2^q C^q K^q} \sum_{j=1}^\infty 2^{-jq} s_{m_j}^q\\
&=\frac{t^q}{2^q C^q K^q} \sum_{j=1}^\infty 2^{-jq} \sum_{k=1}^j\sum_{n=J_k} w_n' \\
&=\frac{(2^q-1) t^q }{C^q K^q} \sum_{k=1}^\infty 2^{-kq} \sum_{n=J_k} w_n' \\
&=\frac{(2^q-1) }{2^q C^q K^q} \sum_{n=1}^\infty |a_n|^q w_n'.
\end{align*}
Using \eqref{eq:rescale}, and taking into account that the concavity constants and the unconditionality constants of $(\xx_{\pi(n)})_{n=1}^\infty$ are still $C$ and $K$ for any permutation $\pi$ of $\NN$, we are done.
\end{proof}

We also need the dual property of URP. A sequence $(\Phi(m))_{m=1}^\infty$ of positive scalars is said to have the \emph{lower regularity property} (LRP for short) if there is $r\in\NN$ such that
\[
\Phi(rn)\ge 2 \Phi(n)\, \quad n\in\NN.
\]
$(\Phi(m))_{m=1}^\infty$ has the LRP if and only if $(m/\Phi(m))_{m=1}^\infty$ has the URP. Hence a dual inequality of \eqref{eq:URP:Regular} holds, i.e., for any sequence $(\Phi(m))_{m=1}^\infty$ with the LRP there is a constant $C$ such that
\begin{equation*}
\sum_{n=1}^m \frac{\Phi(n)}{n} \le C\Phi(m), \quad m\in\NN.
\end{equation*}

\begin{lemma}\label{lem:ConcaveLRP}
Suppose that a sequence space $\LL$ on a set $\Jt$ is $q$-concave for some $0<q<\infty$. The $\varphi_l[\LL]$ has the LRP.
\end{lemma}

\begin{proof}
Let $r$, $m\in\NN$, and $A\subseteq\Jt$ with $|A|=rm$. Pick a partition $(A_j)_{j=1}^r$ of $A$ with $|A_j|=m$ for all $j=1$, \dots $r$. If $C$ is the $q$-concavity constant of $\LL$,
\[
\Vert \Ind_A\Vert_\LL\ge \frac{1}{C}\left( \sum_{j=1}^q \Vert \Ind_{A_j} \Vert^q_\LL\right)^{1/r} \ge \frac{r^{1/q}}{C} \varphi_l[\LL](m).
\]
Hence, if we pick $r \ge (2C)^q$ we get $\varphi_l[\LL](rm)\ge 2 \varphi_l[\LL](m)$.
\end{proof}

Even without having any information on the concavity of the space $X$, Proposition~\ref{cor:SADemConcave} below
provides an improvement of \cite{AlbiacAnsorena2020b}*{Proposition~5.6}. In addition to that, it shows that imposing some nontrivial concavity to the lattice structure allows to weaken the assumption on the lower democracy function.

\begin{proposition}\label{cor:SADemConcave}
Let $\XX$ be a quasi-Banach space with a normalized unconditional basis $\XB$. Suppose that $\XB$ induces a $q$-concave lattice structure,
$1\le q\le \infty$. Denote by $q'$ the conjugate exponent of $q$, and put $s_{m}= \varphi_l[\XB](m)$ for all $m\in\NN$.
Suppose that either $q=1$ and
\[
\lim_m \frac{s_m}{m}=\infty,
\]
or $q>1$ and
\[
\sum_{m=1}^\infty \frac{m^{q'-1}}{s_m^{q'}}<\infty.
\]
Then $\XB$ is strongly absolute.
\end{proposition}

\begin{proof}
If $1\le q<\infty$, applying Lemma~\ref{lem:ConcaveLRP} gives a constant $C$ such that
\begin{equation*}
\sum_{n=1}^m \frac{s_n}{n} \le C s_m, \quad m\in\NN.
\end{equation*}
Set $\ww=s_n/n)_{n=1}^\infty$ if $1<q<\infty$, and let $\ww$ be the weight whose primitive weight is $(s_m)_{m=1}^\infty$ if $q\in\{1,\infty\}$.
By Proposition~\ref{prop:LorentzSA}, the unit vector system of $d_{1,q}(\ww)$ is strongly absolute. Then, the result follows from combining Theorem~\ref{thm:embLorentz} with Lemma~\ref{lem:SADom}.
\end{proof}

\begin{example}\label{ex:Lorentz} Let $\ww=(w_n)_{n=1}^\infty$ be a weight whose primitive weight $(s_n)_{n=1}^\infty$ is doubling.
\begin{enumerate}[label=(\roman*), leftmargin=*, widest=iii]
\item\label{ex:Lorentz:1} If $\sum_{n=1}^\infty 1/s_n <\infty$, the spaces
\begin{align*}
\ell_{p}(\ls_{1,\infty}(\ww))&=(\ls_{1,\infty}(\ww)\oplus\cdots \oplus\ls_{1,\infty}(\ww)\oplus\cdots)_{\ell_p},\\
\ls_{1,\infty}(\ww)(\ell_{p})&=(\ell_{p}\oplus\cdots \oplus\ell_{p}\oplus\cdots)_{\ls_{1,\infty}(\ww)}\end{align*}
have a (UTAP) unconditional basis for all $0<p<1$.
\item\label{ex:Lorentz:2} Let $1<q<\infty$ and denote by $q'$ its conjugate exponent. Suppose that $\sum_{n=1}^\infty w_n^{-q'+1} s_n^{-1} <\infty$. Then the spaces
\begin{align*}
\ell_{p}(d_{1,q}(\ww)&=(d_{1,q}(\ww)\oplus\cdots \oplus d_{1,q}(\ww)\oplus\cdots)_{\ell_p},\\
d_{1,q}(\ww)(\ell_{p})&=(\ell_{p}\oplus\cdots \oplus\ell_{p}\oplus\cdots)_{d_{1,q}(\ww)}
\end{align*}
have a (UTAP) unconditional basis for all $0<p<1$.

\item If $\inf_m s_n/n>0$, then the space
\[
\ell_{p}(d_{1,q}(\ww))=(d_{1,q}(\ww)\oplus\cdots \oplus d_{1,q}(\ww)\oplus\cdots)_{\ell_{p}}
\]
has a (UTAP) unconditional basis for all $0<p<1$ and $0<q<1$. Recall that $d_{1,q}(\ww)$ is, for a different weight $\ww'$, the classical space $d(q,\ww')$, considered in \cite{AlbiacLeranoz2008} (see \ref{eq:LorentzEmb}).

\item\label{ex:Lorentz:3} Let $0<q\le 1$ and suppose that $\lim_n s_n/n=\infty$. Then the spaces $\ell_{p}(d_{1,q}(\ww))$ and $ d_{1,q}(\ww)(\ell_{p})$ have
a (UTAP) unconditional basis for all $0<p\le 1$.

\end{enumerate}
\end{example}

\subsection{Orlicz sequence spaces}\label{ex:Orlicz}
A \emph{normalized Orlicz function} is a right-continuous increasing function $F\colon[0,\infty)\to[0,\infty)$ such $F(0)=0$ and $F(1)=1$. The topological vector space built from the modular
\[
(a_n)_{n=1}^\infty \mapsto \sum_{n=1}^\infty F(|a_n|)
\]
is the Orlicz sequence space usually denoted by $\ell_F$. The space $\ell_F$ is locally bounded if and only if
If there is $p>0$ such that
\begin{equation}\label{eq:OrliczConvex}
\sup_{0<x,t\le 1} \frac{F(tx)}{x^p F(t)}<\infty
\end{equation}
(see \cite{Kalton1977}*{Proposition 4.2}), in which case $\ell_F$ is complete, i.e., $\ell_F$ is a quasi-Banach space equipped with the quasi-norm
\[
\Vert (a_{n})_{n=1}^{\infty}\Vert_{F} = \inf\left\{\rho>0\colon\sum_{n=1}^{\infty} F(|a_{n}|/\rho)\le 1 \right\}.
\] Moreover, if \eqref{eq:OrliczConvex} holds for a given $p$, then the unit vector system induces a $p$-convex lattice structure on $\ell_F$.

Let $G$ be the ``dual'' function of $F$, defined by
$G(t)=t/F(t)$ for $0<t<\infty$.
A standard argument gives that \eqref{eq:OrliczConvex} holds for some $p$ if and only if $G$ is doubling near the origin, i.e., there is a constant $C$ such that
\[
G(t)\le C G(t/2), \quad 0<t\le 1.
\]
Summing up, if $G$ is doubling near the origin, then $\ell_F$ is an $L$-convex (symmetric) sequence space.
The sequence space $\ell_F$ is minimal if and only $F$ is doubling near the origin. Moreover $\ell_F$ is contained in $\ell_1$ if and only if $\inf_{0\le t \le 1} F(t)/t>0$.

Given a normalized Orlicz such $F$ we define its inverse by
\[
H(s)=\sup \{ t\in[0,\infty) \colon F(t) \le s\}, \quad 0\le s<\infty,
\]
If the dual function $G$ is doubling near the origin, then $\ell_F$ is a symmetric sequence space.

In order to apply Theorems~\ref{thm:UWCBSSums} and \ref{thm:l1sumK21} to Orlicz squence spaces it is convenient to have a criterion that guarantees that the unit vector system of $\ell_F$ is strongly absolute, which will imply that it is uniformly universal for well complemented block basic sequences. In some particular cases, the strong absoluteness of the unit vector system of $\ell_F$ can be derived from Proposition~\ref{cor:SADemConcave}. However, using specific techniques for this type of spaces allows to obtain better results. For instance, we will next show that the unit vector system of most Orlicz sequence spaces contained in $\ell_1$ is strongly absolute.

\begin{proposition}\label{prop:OrliczSA}
Let $F$ be a normalized Orlicz function and set $G(t)=t/F(t)$, $0<t<\infty$. Suppose that $G$ doubling near the origin, essentially increasing, and satisfies $\lim_{t\to 0^+}G(t)=0$. Then $F$ is doubling near the origin, the Orlicz sequence space $\ell_F$ is minimal, and the canonical basis is strongly absolute.
\end{proposition}
\begin{proof}Let $C\in[1,\infty)$ be such that
$
G(s)\le CG(t)
$
for all $0<s\le t\le 1$. Since
\[
F(t)=\frac{t}{G(t)}\le \frac{Ct}{G(s)} =\frac{Ct}{s} F(s), \quad 0<s\le t\le 1,
\]
$F$ is doubling near the origin.

Fix $R<\infty$ and pick $\delta>0$ such that $G(t)\le 1/(RC)$ for every $0<t\le \delta$. Given $f=(a_n)_{n=1}^\infty\in\ell_F$, set $u=\Vert f \Vert_\infty$ and $v=\Vert f \Vert_{\ell_F}$. Then
\begin{align*}
w:=\sum_{n=1}^\infty |a_n|
&=\sum_{n=1}^\infty v F\left( \frac{|a_n|}{v} \right) G \left( \frac{|a_n|}{v} \right)\\
&\le Cv G\left( \frac{u}{v} \right) \sum_{n=1}^\infty F\left( \frac{|a_n|}{v} \right)\\
&\le Cv G\left( \frac{u}{v} \right).
\end{align*}
In the case when $u/v\le \delta$ we have $w\le v/R$. Otherwise,
\[
w\le Cv G\left( \frac{u}{v} \right) = \frac{Cu}{ F(u/v)}\le \frac{Cu}{ F(\delta)}.
\]
Hence, $\ell_F$ is strongly absolute with function $\alpha$ given by
\[
\alpha(R)=\frac{C}{F(\inf \{ t \colon RCG(t) <1 \})}, \quad 0<R<\infty.\qedhere
\]
\end{proof}

Let us mention that Kalton \cite{Kalton1977} (implicitly) proved that if $F$ and its dual function $G$ are doubling near the origin, $G$ is bounded near the origin, and
\begin{equation*}
\lim_{\varepsilon\to 0^+} \inf_{0<s<1}\frac{-1}{\log \varepsilon}\int_\varepsilon^1 \frac{F(sx)}{sx^2}\, dx=\infty,
\end{equation*}
then the unit vector system of the minimal sequence space $\ell_F$ has the peaking property. It is unclear whether these assumptions imply that $\ell_F$ has the uniform peaking property.

\begin{example}\label{ex:Orlicz1}
Let $F$ be a normalized Orlicz function and for $t>0$ let $G(t)=t/F(t)$. Assume that $F$ and $G$ are doubling near the origin, that $G$ is essentially increasing, and that $\lim_{t\to 0^+} G(t)=0$. Then the following spaces have a (UTAP) unconditional basis:
\begin{enumerate}[label=(\roman*), leftmargin=*, widest=iii]
\item $\ell_{1}(\ell_{F})=(\ell_{F}\oplus \ell_{F}\oplus\cdots \oplus \ell_{F}\oplus\cdots)_{1}$;
\item\label{ex:Orlicz:2} $(\bigoplus_{n=1}^\infty \ell_q^n)_{\ell_F}=( \ell_{q}^{1}\oplus \ell_{q}^{2}\oplus\cdots \oplus \ell_{q}^{n}\oplus\cdots)_{\ell_{F}}$ for all $0<q\le 1$;
\item\label{ex:Orlicz:3} $\ell_{F}(c_{0})=(c_{0}\oplus c_{0}\oplus\cdots \oplus c_{0}\oplus\cdots)_{\ell_{F}}$;
\item $\ell_{F}(\ls_{1,\infty}(\ww))=(\ls_{1,\infty}(\ww)\oplus \ls_{1,\infty}(\ww)\oplus\cdots \ls_{1,\infty}(\ww)\dots)_{\ell_{F}}$ and
\item[] $\ls_{1,\infty}(\ww)(\ell_{F})=(\ell_{F}\oplus \ell_{F}\oplus\cdots \ell_{F}\dots)_{\ls_{1,\infty}(\ww)}$, where $\ww$ is as in Example~\ref{ex:Lorentz}~\ref{ex:Lorentz:1}.
\item $\ell_{F}(d_{1,q}(\ww)=(d_{1,q}(\ww)\oplus d_{1,q}(\ww)\oplus\cdots d_{1,q}(\ww)\dots)_{\ell_{F}}$ and
\item[] $d_{1,q}(\ww)(\ell_{F})=(\ell_{F}\oplus \ell_{F}\oplus\cdots \ell_{F}\dots)_{d_{1,q}(\ww)}$, where $\ww$ and $q$ are as in Example~\ref{ex:Lorentz}~\ref{ex:Lorentz:2} and \ref{ex:Lorentz:3}.

\end{enumerate}
\end{example}

For instance, the uniqueness of unconditional basis in $\ell_{1}(\ell_{F})$ is an application of Theorem~\ref{thm:l1SAUTAP}. To see \ref{ex:Orlicz:2} when $0<q<1$, we just need to apply Theorem~\ref{WojtCriterion} since condition~\ref{it:SubSym} in Lemma~\ref{lem:squares} is fulfiled. Then by Proposition~\ref{NewStronglyAbsBases}, the canonical basis of $\ell_{F}(\ell_{q})$ is strongly absolute and equivalent to its square. To show the case when $q=1$ in part \ref{ex:Orlicz:2} and part \ref{ex:Orlicz:3} however, we need to appeal to Theorem~\ref{thm:UTAPLSB} and take into account that the unit vector basis of $\ell_{1}$ and $c_{0}$ is perfectly homogeneous, hence uniformly universal for well complemented block basic sequences with function $C\mapsto 1$. The verification of the corresponding hypotheses leading to the uniqueness property in the remaining cases is totally straightforward, and so we leave it for the reader.

\subsection{Bourgin-Nakano spaces} A \textit{Bourgin-Nakano index} is a family $(p_n)_{n \in \NN}$ in $(0,\infty)$ with $p=\inf_n p_n>0$. The \textit{Bourgin-Nakano space} $\ell(p_{n})$ is the quasi-Banach space built from the modular
\[
m_{(p_n)}\colon\FF^{\NN} \to[0,\infty), \quad (a_n)_{n \in \NN} \mapsto \sum_{n \in \NN} |a_n|^{p_n}.
\]
If we endow $\ell(p_{n})$ with the natural ordering, it becomes a $p$-convex quasi-Banach lattice. The separable part $h(p_{n})=[\ee_n \colon n \in \Nt]$ of $\ell(p_{n})$ is a minimal sequence space. We have $\ell(p_{n})=h(p_{n})$ if and only if $\sup_n p_n<\infty$.

The unit vector system $(\ee_{n})_{n=1}^{\infty}$ of $\ell(p_{n})$ is a $1$-unconditional basis which is universal for well complemented block basic sequences (\cite{AlbiacAnsorena2020b}*{Proposition 4.7}), and which is strongly absolute if and only if $q:=\limsup p_{n}< 1$. Indeed, this condition implies that the space embeds naturally into $\ell_{q}$ and so we can apply Lemma~\ref{lem:SADom}. Moreover, the Banach envelope of $\ell(p_{n})$ is anti-Euclidean if and only if $\limsup p_{n}\le 1$ (see \cite{AlbiacAnsorena2020b}*{Proposition 4.5 and Corollary 4.6}). Combining \cite{AlbiacAnsorena2020b}*{Lemma 4.3} with Lemma~\ref{lem:k2one} gives that the unit vector system of any Bourgin-Nakano space is uniformly universal for well-complemented block basic sequences with function $C\mapsto 4C^2$.

\begin{example} The following spaces have a (UTAP) unconditional basis:
\begin{enumerate}[label=(\roman*), leftmargin=*, widest=iii]
\item $\ell(p_{n})(\ell_{1})=\{(z_{n})_{n=1}^{\infty}\colon z_{n}\in \ell_{1} \;\text{and}\; \sum_{n=1}^{\infty}\Vert z_{n}\Vert_{1}^{p_{n}}<\infty\}$, where $\limsup p_{n}<1$.

\item\label{ex:Nakano:2} $\ell_{F}(\ell(p_{n}))=(\ell(p_{n})\oplus \ell(p_{n})\oplus\cdots \oplus \ell(p_{n})\oplus\cdots)_{\ell_{F}}$, where $F$ is an Orlicz function as in Example~\ref{ex:Orlicz1} and $\limsup p_{n}\le1$.
\item\label{ex:Nakano:3} $\ell_{1}(\ell(p_{n}))=(\ell(p_{n})\oplus \ell(p_{n})\oplus\cdots \oplus \ell(p_{n})\oplus\cdots)_{\ell_q}$, where $\lim p_{n}< 1$.
\end{enumerate}
\end{example}

Indeed, the uniqueness of unconditional basis in the first example follows from a direct application of Theorem~\ref{thm:UTAPLSB}, where we use that the basis of $\ell_{1}$ is equivalent to its square (condition \ref{UTAPLSB:SQ}~\ref{UTAPLSB:SQ:1}). In the second example we use also Theorem~\ref{thm:UTAPLSB}, but now we employ condition \ref{UTAPLSB:SQ}~\ref{UTAPLSB:SQ:2} since, while $\ell(p_{n})$ need not be lattice isomorphic to its square, $\ell_F$ is.
The last case is just a direct application of Theorem~\ref{thm:l1SAUTAP}, since the hypothesis ensures that the canonical basis of $\ell(p_{n})$ is strongly absolute. Note that in the cases \ref{ex:Nakano:2} and \ref{ex:Nakano:3}, the uniqueness of unconditional basis in the direct sum is obtained without knowing whether the space $\ell(p_{n}))$ has a unique unconditional basis or not!

\subsection{Hardy spaces} Because of their importance in Analysis, we single out as well some examples involving Hardy spaces. For the convenience of the reader we will next state a few known facts about the spaces $H_p(\TT^d)$ that we will need in order to apply Theorems~\ref{thm:UTAPLSB} and \ref{thm:l1SAUTAP}.

The first unconditional bases in $H_p(\TT)$ for $0<p<1$ were constructed in \cite{Woj1984}. Those bases allow a manageable expression for the norm in terms of the coefficients relative to the basis. Namely, if $\HB= (\xx_n)_{n=0}^\infty$ is such a normalized basis then
\begin{equation}\label{dyadic}
\left\|\sum_{n=0}^\infty a_n \, \xx_n\right\|_{H_p(\TT)}\approx \left( \int_0^1\left(\sum_{n=0}^\infty|a_n|^2 h_n^2\right)^{p/2}\right)^{1/p}, \ (a_n)_{n=1}^\infty\in c_{00},
\end{equation}
where $(h_n)_{n=0}^\infty$ is the classical Haar system on $[0,1]$ normalized with respect to the norm in $L_p([0,1])$. Those bases allow tensor constructions of unconditional bases in $H_p(\TT^d)$ for $d\in \NN$ which satisfy an equivalence analogous to \eqref{dyadic}. Using those tensored bases, Kalton et al.\ \cite{KLW1990} showed that the spaces $H_p(\TT^d) $ and $H_p(\TT^{d'})$ with $0<p<1$ and $d$, $d'\in \NN$, are isomorphic if and only if $d=d'$. Then it was proved in \cite{Woj1997} that all the spaces $H_p(\TT^d)$ for $0<p<1$ and $d\in \NN$ have a UTAP unconditional basis.

The canonical basis $\HB$ of the Hardy spaces $H_p(\TT^d)$, $0<p<1$, induces a $p$-convex lattice structure and satisfies the estimate
\[
m^{1/p} \approx \varphi_m^l[\HB, H_p(\TT^d)],\quad m\in\NN.
\]
Hence, Proposition~\ref{cor:SADemConcave} implies that $\HB$ is strongly absolute This way we can use Hardy spaces (or more generally subspaces of Hardy spaces generated by subbases of the Haar system) to build examples of spaces with a (UTAP) unconditional basis. Given a (finite or infinite) nonempty subset $\nn\subseteq\NN$, $H_p^\nn(\TT)$ denotes the subspace of $H_p(\TT)$ generated by Haar functions belonging to layers in $\nn$.

\begin{example} The following spaces have a (UTAP) unconditional basis:
\begin{enumerate}[label=(\roman*), leftmargin=*, widest=iii]
\item The space $H_{p}(\TT,d_{1,q}(\ww))=d_{1,q} \oplus d_{1,q}(\ww)\oplus\cdots\oplus d_{1,q}(\ww)\oplus\cdots)_{H_{p}}$, consisting of all sequences $(\zz_{n})_{n=1}^{\infty}$ such that $\zz_{n}\in d_{1,q}(\ww)$ for all $n\in \NN$ and
\[
\left( \int_0^1\left(\sum_{n=0}^\infty\Vert\zz_n\Vert^2 h_n^2(t)\right)^{p/2}dt\right)^{1/p}<\infty,
\] where $p<1$, $0<q< 1$ and the primitive weight $(s_m)_{m=1}^\infty$ of $\ww$ satisfies $\inf_m s_m/m>0$.

\item $\ell_{F}(H_{p}(\TT))=(H_{p}(\TT)\oplus H_{p}(\TT)\oplus\cdots\oplus H_{p}(\TT)\oplus\cdots)_{\ell_{F}}$, where $F$ is as in Proposition~\ref{prop:OrliczSA}.

\item $(\bigoplus_{k=1}^\infty H_p^{\nn_k}(\TT))_{\ell_1}=(H_{p}^{\nn_1}(\TT)\oplus H_{p}^{\nn_2}(\TT)\oplus\cdots\oplus H_{p}^{\nn_k}(\TT)\oplus\cdots)_{\ell_1}$, where $0<p<1$ and $(\nn_k)_{k=1}^\infty$ is an increasing sequence of subsets of $\NN$.

\end{enumerate}

Note that in the last example, since there are sets of layers $\nn\subseteq\NN$ for which $H_p^{\nn}(\TT)$ is not isomorphic to its square (see \cite{Woj1997}), we must use condition \ref{l1SAUTAP:SQ}~\ref{l1SAUTAP:SQ:2} in order for all the hypotheses of Theorem~\ref{thm:l1SAUTAP} to be satisfied. As a matter of fact, it is unknown whether $H_p^{\nn}(\TT)$ has a UTAP unconditional basis in the case when it is not isomorphic to its square (see Theorem~\ref{WojtCriterion}).

\end{example}

\subsection{Tsirelson's space}Casazza and Kalton established in \cite{CasKal1998} the uniqueness of unconditional basis up to permutation of Tsirelson's space $\Ts$ and its complemented subspaces with unconditional basis as a byproduct of their study of complemented basic sequences in lattice anti-Euclidean Banach spaces. Their result answered a question by Bourgain et al.\ (\cite{BCLT1985}), who had proved the uniqueness of unconditional basis up to permutation of the $2$-convexifyed Tsirelson's space $\Ts^{(2)}$. Unlike $\Ts^{(2)}$, which is ``highly" Euclidean, the space $\Ts$ is anti-Euclidean. To see the latter requires the notion of dominance, introduced in \cite{CasKal1998}.

Let $\XB=(\xx_n)_{n=1}^\infty$ be a (normalized) unconditional basis of a quasi-Banach space $\XX$. Given $f$, $g\in\XX$, we write $f\prec g$ if $m<n$ for all $m\in\supp(f)$ and $n\in\supp(g)$. Given $D\ge 1$, the basis $\XB$ is said to be \emph{left (resp.\ right) $D$-dominant} if whenever $(f_i)_{i=1}^n$ and $(g_i)_{i=1}^n$ are disjointly supported families with $f_i\prec g_i$ (resp.\ $g_i\prec f_i$) and $\Vert f_i\Vert \le \Vert g_i\Vert$ for all $i=1$, \dots, $n$, then $\Vert \sum_{i=1}^n f_i\Vert \le D \Vert \sum_{i=1}^n g_i\Vert$. As is customary, if the constant $D$ is irrelevant, we just drop it from the notation. If $\XX$ is a Banach space with a left (resp.\ right) dominant semi-normalized unconditional basis $\XB$ there is a unique $r=r(\XB)\in[1,\infty]$ such that $\ell_r$ is finitely block representable in $\XX$. In the case when $r(\XB)\in\{1,\infty\}$, $\XX$ is anti-Euclidean (see \cite{CasKal1998}*{Proposition 5.3}).

The canonical basis of the \textit{Tsirelson space} $\Ts$ is (normalized, $1$-unconditional and) right $16$-dominant (see \cite{CasKal1998}*{Proposition 5.12}) with $r(\Ts)=1$. In turn, by \cite{CasKal1998}*{Lemma 5.1}, the canonical basis of the original Tsirelson's space $\Ts^{\ast}$ is left-dominant.

Moreover, by \cite{CasKal1998}*{Proposition 5.5} and \cite{CasShu1989}*{page 14}, the canonical bases of $\Ts$ and $\Ts^{\ast}$ (as well as each of their subases) are equivalent to their square. In our language, \cite{CasKal1998}*{Theorem 5.6} says that every left (resp.\ right) dominant unconditional basis is universal for well complemented block basic sequences. Combining the arguments used in its proof with Lemma~\ref{lem:k2one} yields the following quantitative result.

\begin{theorem}
Let $\XB=(\xx_k)_{k=1}^\infty$ be a left (or right) $D$-dominant normalized $K$-unconditional basis of a quasi-Banach space $\XX$ with modulus of concavity at most $\kappa$. Then $\XB$ is uniformly universal for well complemented block basic sequences with function depending on $D$, $K$ and $\kappa$.
\end{theorem}

\begin{proof} Let us do the right-dominant case only as the left-dominant case is similar. For that, we first show that there are constants $D_1$ and $D_2$ depending only on $D$, $\kappa$ and $K$ such that any semi-normalized disjointly supported basic sequence $\UB=(\uu_m)_{m\in\Mt}$ $D_2$-dominates
$(a\,\xx_{k_m})_{m\in\Mt}$ and it is $D_1$-dominated by $(b\, \xx_{j_m})_{m\in\Mt}$, where
$a=\inf_m\Vert \uu_m\Vert$, $b=\sup_{m} \Vert \uu_m\Vert$, $
j_m=\min(\supp(\yy_m))$, and $k_m=\max(\supp(\yy_m))$.

Indeed, if $A_m=\supp(\uu_m)\setminus\{ j_m\}$, and $(a_m)_{m\in\Mt}\in c_{00}(\Mt)$,
\begin{align*}
\Big\Vert \sum_{m\in\Mt} & a_m\, \uu_m \Big\Vert
\le \kappa\left( \left\Vert \sum_{m\in\Mt} a_m \, S_{A_m}(\uu_m)\right\Vert+ \left\Vert \sum_{m\in\Mt} a_m \, \xx_{j_m}^*(\uu_m)\, \xx_{j_m}\right\Vert\right)\\
&\le \kappa D \left( \left\Vert \sum_{m\in\Mt} a_m \, \Vert S_{A_m}(\uu_m)\Vert \xx_{j_m}\right\Vert+ \left\Vert \sum_{m\in\Mt} a_m \, \xx_{j_m}^*(\uu_m)\, \xx_{j_m}\right\Vert\right) \\
&\le 2 \kappa K D b \left\Vert \sum_{m\in\Mt} a_m \, \xx_{j_m}\right\Vert.
\end{align*}
In turn, if $F_m=\supp(\uu_m)\setminus\{ k_m\}$, there are $(\lambda_m)_{m\in\Mt}$ and $(\gamma_m)_{m\in\Mt}$ such that $a=\kappa(\lambda_m+\gamma_m)$, $0\le \lambda_m\le \Vert S_{F_m}(\uu_m)\Vert$, and $0\le\gamma_m\le |\xx_{k_m}^*(\uu_m)|$ for all $m\in\Mt$. Hence,
\begin{align*}
\Big\Vert \sum_{m\in\Mt} a_m\, & \uu_m \Big\Vert\\
&\ge \frac{1}{K} \max\left\{ \left\Vert \sum_{m\in\Mt} a_m \, \lambda_m \frac{S_{F_m}(\uu_m)}{\Vert S_{F_m}(\uu_m)\Vert }\right\Vert,
\left\Vert \sum_{m\in\Mt} a_m \, \gamma_m\, \xx_{k_m}\right\Vert\right\}\\
&\ge \frac{1}{KD} \max\left\{ \left\Vert \sum_{m\in\Mt} a_m \, \lambda_m\, \xx_{k_m} \right\Vert,
\left\Vert \sum_{m\in\Mt} a_m \, \gamma_m\, \xx_{k_m}\right\Vert\right\}\\
&\ge \frac{a}{2\kappa^2 KD} \left\Vert \sum_{m\in\Mt} a_m\, \xx_{k_m} \right\Vert.
\end{align*}

Pick $0<\lambda<1$.
Let $\YB=(\yy_m)_{m\in\Mt}$ be a well complemented normalized basic sequence with good $C$-projecting functionals $\YB=(\yy^*_m)_{m\in\Mt}$. For each $m\in\Mt$ there is $k_m\in\supp(\yy_m)$ such that, if $A^{l}_m=\supp(\yy_m)\cap[1,k_m]$ and $A^{r}_m=\supp(\yy_m)\cap[k_m,\infty)$,
\[
|\yy_m^*(S_{A^l_m}(\yy_m))| \ge \lambda, \quad |\yy_m^*(S_{A^r_m}(\yy_m))|\ge 1-\lambda.
\]
By Lemma~\ref{lem:k2one}, $\YB^r:=(S_{A^r_m}(\yy_m))_{m\in\Mt}$ $(CK/(1-\lambda))$-dominates $\YB$, and $\YB$ $(CK)$-dominates $\YB^l:=(S_{A^l_m}(\yy_m))_{m\in\Mt}$. Moreover $\YB^l$ $(CK/\lambda)$-dominates $\YB$, whence
\[
\Vert S_{A^l_m}(\yy_m)\Vert\ge \frac{\lambda}{CK} , \quad m\in\Mt.
\]
Therefore, $\YB^l$ $(CKD_2/\lambda)$-dominates $(\xx_{k_m})_{m\in\Mt}$.
Since $\Vert S_{A^r_m}(\yy_m)\Vert\le K$ for all $m\in\Mt$, $(\xx_{k_m})_{m\in\Mt}$ $(KD_1)$-dominates $\YB^r$. Summing up, choosing $\lambda=1/(1+\kappa C K)$ we infer that $\XB$ is uniformly universal for well complemented basic sequences with function
\[
C\mapsto 2\kappa C K^3 D (1+\kappa C).\qedhere
\]
\end{proof}

Finally, since they are locally convex, both $\Ts$ and $\Ts^{\ast}$ are trivially $L$-convex lattices.

Combining the above background information with our main results we bring out a couple of examples:

\begin{example} For $0<p<1$ and $F$ an Orlicz function as in Proposition~\ref{prop:OrliczSA},
the spaces
\begin{enumerate}[label=(\roman*), leftmargin=*, widest=iii]
\item $\ell_{p}(\Ts)=(\Ts\oplus \Ts\oplus\cdots\oplus \Ts\oplus\cdots)_{\ell_p}$ and
\item $\ell_{F}(\Ts^{\ast})=(\Ts^{\ast}\oplus \Ts^{\ast}\oplus\cdots\oplus\Ts^{\ast}\oplus\cdots)_{\ell_{F}}$
\end{enumerate}
have a (UTAP) unconditional basis.
\end{example}

\subsection{Mixed-norm Lebesgue sequence spaces}

\noindent We close with applications to finite and infinite direct sums of mixed-norm Lebesgue sequence spaces.

\begin{example}\label{ex:mix:norm:1}
Suppose $(p_j)_{j=1}^n$ is sequence of indexes in $(0,1]$ with $p_j=1$ for at most one $j$. We consider the space
\[
\XX=\ell_{p_1}(\ell_{p_2}(\cdots\ell_{p_i}(\cdots(\ell_{p_n})))
\]
of recursive direct sums of a finite number of (possibly repeated) sequence spaces $\ell_{p_{j}}$. $\XX$ is a $p$-Banach space for $p=\min_{j} p_{j}$, and its canonical basis $\XB$ is unconditional and equivalent to its square. Moreover, $\XB$ induces on $\XX$ a $p$-convex lattice structure, and it dominates the unit vector system of $\ell_q$, where $q:=\max_j p_j$. Thus, in the case when $q<1$, Lemma~\ref{lem:SADom} implies that $\XB$ is also strongly absolute. Therefore, by Theorem~\ref{WojtCriterion}, $\XX$ has a (UTAP) unconditional basis. If we let one (and only one) of the indexes $p_{j}$ be 1, we need to distinguish two cases. Suppose first that $p_{1}=1$ and $0<p_{j}<1$ for $1<j\le n$. Then, as before, the canonical basis of $\ell_{p_2}(\ell_{p_3}(\cdots\ell_{p_i}(\cdots(\ell_{p_n})))$ is strongly absolute and so the uniqueness of unconditional basis of $\ell_{1}(\ell_{p_2}(\cdots\ell_{p_i}(\cdots(\ell_{p_n})))$ follows from Theorem~\ref{thm:l1SAUTAP}. Now, suppose that an index other than $p_{1}$, say $p_{3}$, is equal to $1$. On one hand, the canonical basis of $\ell_{p_4}(\ell_{p_5}(\cdots\ell_{p_i}(\cdots(\ell_{p_n})))$ is strongly absolute, so that by Theorem~\ref{thm:l1sumK21}, the canonical basis of $\ell_{p_{3}}(\ell_{p_4}(\ell_{p_5}(\cdots\ell_{p_i}(\cdots(\ell_{p_n})))$ is uniformly universal for well-complemented block basic sequences. On the other hand, since the canonical basis of $\ell_{p_{1}}(\ell_{p_{2}})$ is strongly absolute, we just need to apply Theorem~\ref{thm:UTAPLSB}.
\end{example}

\begin{remark}
In Example~\ref{ex:mix:norm:1}, the hypothesis that at most one $p_{j}=1$ is important for the validity of the uniqueness result. For instance, we don't known whether $\ell_1(\ell_p(\ell_1)))$ has a UTAP unconditional bases.
\end{remark}

\begin{example}\label{ex:mix:norm:2}
Let $(p_n)_{n=1}^\infty$ a one-to-one sequence in $(0,1]$ with $s:=\inf_n p_n>0$, and let $q\in (0,1]$. Consider now the space
\[
\XX=\left(\bigoplus_{n=1}^\infty \ell_{p_n}\right)_{\ell_q}=(\ell_{p_{1}}\oplus\ \ell_{p_{2}}\oplus\cdots\oplus\ell_{p_{n}}\oplus\cdots)_{\ell_q}.
\]
Note that since $\ell_p(\ell_p)=\ell_p$ isometrically for all $p>0$, there is no real restriction in assuming that the indices $p_j$ are not repeated.

The unit vector system $\EE_p$ of $\ell_p$ is $2^{1/p}$ equivalent to its square. Moreover, $\EE_p$ is perfectly homogeneous, thus uniformly universal for well complemented block basic sequences with function $C\mapsto 1$. Finally, $\EE_p$ is $1$-unconditional and, if we consider on $\ell_p$ the lattice structure induced by $\EE_p$, $M_r(\ell_p)=1$ for all $r\le p$. Hence, in the case when $q<1$, the uniqueness of unconditional basis of $\XX$ is an application of Theorem~\ref{thm:UTAPLSB}, where the hypothesis \ref{UTAPLSB:SQ} is fulfilled with condition \ref{UTAPLSB:SQ:1}.

 Suppose now that $q=1$ and $t:=\sup_n p_n<1$. The important detail here is that all the canonical bases of $\ell_{p_{n}}$ are strongly absolute with the same function $\alpha$. In fact, by \cite{AAW2021}*{Lemma 3.2}, we can choose
\[
\alpha(R)=\begin{cases} R^{t/(1-t)} & \text{ if }R\ge 1,\\ R^{s/(1-s)} & \text{ if }R\le 1.\end{cases}
\]
Hence, applying Theorem~\ref{thm:l1SAUTAP} gives that $\XX$ has a (UTAP) unconditional basis.
\end{example}

\begin{remark}
In Example~\ref{ex:mix:norm:2},
We do not known whether $\XX$ has a UTAP unconditional basis in the case when $q=1$ and $\lim_n p_n=1$.
\end{remark}



\end{document}